\documentclass[12pt,a4paper]{article}
\usepackage[latin1]{inputenc}
\usepackage{amsmath, amsthm, amssymb, amsfonts, amscd}
\usepackage[english]{babel}
\usepackage[all]{xy}
\usepackage[linktocpage=true]{hyperref}
\usepackage{color}
\usepackage[table]{xcolor}
\usepackage{pinlabel}
\usepackage{tikz-cd}
\usepackage{booktabs}
\usepackage{enumitem}
\usepackage{geometry}
\usepackage{mathtools}
\usepackage{authblk}
\usepackage{thmtools}
\usepackage{stmaryrd}
\usepackage{comment}
\usepackage{lipsum}

\declaretheoremstyle[bodyfont=\normalfont,spaceabove=\medskipamount,
    spacebelow=\medskipamount]{definition}

\theoremstyle{definition}

\newtheorem{theorem}{Theorem}[section]
\newtheorem{lemma}[theorem]{Lemma}
\newtheorem{corollary}[theorem]{Corollary}
\newtheorem{proposition}[theorem]{Proposition}
\newtheorem{definition}[theorem]{Definition}

\newtheorem{remark}[theorem]{Remark}

\newtheorem{example}[theorem]{Example}

\newcommand{\defeq}{\vcentcolon=}

\title{\large \textbf{THISTLETHWAITE THEOREMS FOR KNOTOIDS AND LINKOIDS}}

\author{\normalsize SERGEI CHMUTOV,  QINGYING DENG, JOANNA A.~ELLIS-MONAGHAN,\linebreak SERGEI LANDO, WOUT MOLTMAKER}

\date{}

\begin{document}

\maketitle

\begin{abstract}
The classical Thistlethwaite theorem for links can be phrased as asserting that the Kauffman bracket of a link can be obtained from an evaluation of the Bollob\'{a}s-Riordan polynomial of a ribbon graph associated to one of the link's Kauffman states. In this paper, we extend this result to knotoids, which are a generalization of knots that naturally arises in the study of protein topology. Specifically 
we extend the Thistlethwaite theorem to the twisted arrow polynomial of knotoids, which is an invariant of knotoids on compact, not necessarily orientable, surfaces. To this end, we define twisted knotoids, marked ribbon graphs, and their arrow- and Bollob\'{a}s-Riordan polynomials. We also 
extend the Thistlethwaite theorem to the loop arrow polynomial of knotoids in the plane, and to spherical linkoids.
\end{abstract}




\section{Introduction}

The classical Thistlethwaite paper \cite{thistlethwaite1987} relates knot theory to combinatorics. More precisely, Thistlethwaite's theorem claims that up to a sign and a power of $t$ factor, the Jones polynomial  $V_L(t)$ of an alternating link $L$ is equal to the Tutte polynomial $T_{G_L}(-t,-t^{-1})$
of the plane graph $G_L$ obtained from a checkerboard coloring of the regions of a link diagram of $L$. Its generalization to arbitrary (not necessarily alternating) links leads to the notion of signed plane graph with signs $\pm$ assigned to the edges of the graph. This requires a signed extension of the Tutte polynomial, which was done by L.Kauffman in \cite{kauffman1989}.
I.Pak suggested to use the Bollob\'as-Riordan polynomial of ribbon graphs 
(as opposed to plane graphs) as a generalization of the Tutte polynomial. This idea was realized first in \cite{chmutov-pak2007} for virtual links with checkerboard colorable diagrams. The
Bollob\'as-Riordan polynomial was used for classical links in \cite{DFKLS} where 
the concepts of {\it Turaev surface}, {\it quasi-trees} etc.~were introduced. To arbitrary virtual links the Thistlethwaite theorem was generalized in \cite{chmutov-voltz2008} using the Seifert state as an initial data to construct the corresponding ribbon graph. An attempt to unify all these approaches led to the general virtual Thistlethwaite theorem \cite{chmutov-2009}, where an arbitrary Kauffman state may be used to construct a ribbon graph. Here a `Kauffman state' is a state in the Kauffman state sum for the bracket polynomial \cite{kauffman1987}. Independence of  the result with respect to the choice of the
initial state leads to the concept of {\it partial duality} of ribbon graphs and the invariance of the Bollob\'as-Riordan polynomial under it. The further generalization of this theorem for the {\it arrow polynomial} instead of the Jones polynomial was done in \cite{bradford2012arrow} and then in \cite{paugh-wu-zhang2021} for twisted links.

In this paper, we extend all these results to knotoids. Knotoids are a natural generalization of knots and links originally defined and studied by Turaev \cite{turaev2012}, in which one allows for knot diagrams with open ends. Unlike for braids and tangles, these endpoints may lie anywhere, including interior regions of the diagram. To ensure these open-ended components are not all trivial, we define consider knotoid diagram up to Reidemeister moves \emph{away from} the endpoints, and in particular it is not allowed to pull an endpoint over or under an adjacent crossing.

After Turaev's initial work \cite{turaev2012}, knotoids were studied further in works such as \cite{gugumcu2017,barbensi2018double,moltmaker2023}, where various strong knotoid invariants were constructed. Knotoids and their invariants have also found applications in modelling the topology of proteins \cite{dorier2018knoto,goundaroulis2020knotoids,panagiotou2020knot}. This work on invariants culminated in the first attempt at tabulating knotoids for low crossing number \cite{goundaroulis2019systematic}, for which the authors used the arrow- and loop arrow polynomials, among other invariants. In this paper we consider these arrow polynomials, giving Thistlethwaite theorems for them in the knotoid setting. This requires the introduction of decorated ribbon graphs which we call {\it marked ribbon graphs}, as well as polynomials associated with these objects. We further generalize our results to the setting of \textit{twisted knotoids}, which we introduce here and which model knotoids on closed (possibly non-orientable) surfaces, thereby generalizing virtual knotoids introduced in \cite{bourgoin2008}. These twisted knotoids are the knotoidal analogue of twisted links, which are links in the orientable thickenings of closed surfaces \cite{bourgoin2008}. To complete this analogy, we show that twisted knotoids similarly correspond to so-called \emph{$H$-curves} in orientable thickenings.

The paper is organized as follows. In Section \ref{sec:prelims} we introduce knotoids and twisted knotoids, and show how they correspond to $H$-curves. In Section \ref{sec:arrow} we recall the arrow- and loop arrow polynomials, and define an arrow polynomial for twisted knotoids. In Section \ref{sec:BR} we define marked ribbon graphs and their Bollob\'{a}s-Riordan polynomials, and give some basic results for these. Putting these together, in Section \ref{sec:Thistlethwaite} we show the Thistlethwaite theorems obtaining the various arrow polynomials as evaluations of Bollob\'{a}s-Riordan polynomials, and give some examples. Finally, in Section \ref{sec:linkoids} we briefly discuss the extension of our work to linkoids, which are multi-component knotoids.




\section{Preliminaries}\label{sec:prelims}

In this section we introduce twisted knotoids, the topological objects of interest to this paper. We also give some basic results for these objects.

\subsection{Knotoids}

Throughout this paper we write $I$ for $[0,1]$.

\begin{definition}\label{def:knotoid}
\cite{turaev2012} Let $\Sigma$ be a surface. A \emph{knotoid diagram} on $\Sigma$ is an immersion $I\to \Sigma$ all whose singularities are transversal double points endowed with over/under-crossing information. We refer to the image of $0$ as the \emph{tail} of a knotoid, and to the image of $1$ as its \emph{head}. Knotoids are always oriented from tail to head. We say two knotoid diagrams are \textit{equivalent} if they can be related by a sequence of ambient isotopies, Reidemeister moves $R1$, $R2$, and $R3$, and self-homeomorphisms of $\Sigma$ that preserve the diagram's orientation. A \emph{knotoid} on $\Sigma$ is an equivalence class of knotoid diagrams on $\Sigma$.
\end{definition}


Note that none of the Reidermeister moves involve endpoints.  Thus, there are no moves that create or destroy a crossing by moving an endpoint over or under an arc. Such moves are also called \emph{forbidden moves}.

\begin{definition}
A \emph{spherical knotoid} is a knotoid on $\Sigma=S^2$, and a \emph{planar knotoid} is a knotoid on $\Sigma=\mathbb{R}^2$.
\end{definition}


\begin{example}\label{ex:1crossing}
Let $K$ be the knotoid diagram with one crossing, shown in Figure \ref{fig:1crossing}. As a planar knotoid $K$ is non-trivial, i.e.~inequivalent to the zero-crossing knotoid, as can be seen by applying for example the loop arrow polynomial defined in Section \ref{sec:arrow}. However, as a spherical knotoid $K$ is trivial. Indeed, the moves that trivialize $K$ are given by moving the exterior arc over the sphere to the other side of the diagram's endpoints and applying an $R1$ move. On the plane this is not possible as the endpoints prevent this arc from moving over them, but on the sphere we can move the arc outward, across the point at $\infty$, and to the other side without incurring any forbidden moves; see Figure \ref{fig:1crossing}.

\begin{figure}[ht]
    \centering
    \includegraphics[width=.6\linewidth]{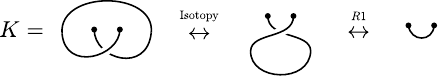}
    \caption{The one-crossing knotoid is trivial on $S^2$.}
    \label{fig:1crossing}
\end{figure}
\end{example}

\begin{remark}\label{rk:punctures}
Example \ref{ex:1crossing} exhibits precisely the difference between planar and spherical knotoids: the ability to move arcs past $\infty$. Thus spherical knotoids are equivalent to knotoids in the plane modulo \emph{spherical moves} which move an exterior arc to the other side of the diagram \cite{moltmaker2022}. Equivalently we can see planar knotoids as spherical knotoids containing a point marked `$\infty$' that arcs may not be passed across, treating planar knotoids as punctured spherical knotoids. For clarity we will use this last convention and draw planar knotoid diagrams as having a puncture.
\end{remark}

\subsection{Twisted knotoids}

Like the knotoids of Definition~\ref{def:knotoid}, \emph{twisted knotoids} are also equivalence classes of diagrams on surfaces, but here they must be on a compact surface (thus not the plane), and we introduce stabilization moves to the equivalence relation.

\begin{definition}
A \emph{twisted knotoid} is an equivalence class of knotoid diagrams on a compact (not necessarily orientable) surface considered up to self-homeomorphism conserving the diagram's orientation and \emph{stable equivalence}.  Stable equivalence is the equivalence relation generated by surface isotopy and the Reidemeister moves, as well as (de)stabilization moves. A \textit{destabilization move} consists of cutting the surface along a closed loop that doesn't intersect the diagram and capping off the resulting boundary circles with disks. (This is with two discs if the loop was orientation-preserving, or with one otherwise.) A \textit{stabilization move} is the reverse of a destabilization move.
\end{definition}

Since twisted knotoids do not subsume planar knotoids, we consider planar knotoids separately in the remainder of this paper.

To represent twisted knotoids as planar diagrams, we can consider decorated knotoid diagrams. The decorations are either virtual crossings or bars. Virtual crossings represent places where the knotoid passes over itself along a handle, and bars represent places where an arc of the knotoid moves through a cross-cap.

\begin{definition}
A \emph{twisted knotoid diagram} is a knotoid diagram on $S^2$ whose crossings may be decorated either as a classical crossing or with a small circle as a \emph{virtual crossing}, familiar from virtual knot theory \cite{kauffman2021}, and whose arcs may be decorated with \emph{bars}, which are small segments drawn transversal to the arc.
\end{definition}

The equivalence relation on twisted knotoids carries over to an equivalence on twisted knotoid diagrams, generated by ambient isotopy and the set of \emph{twisted Reidemeister moves}; see Figure \ref{fig:moves}.

\begin{figure}[ht]
    \centering
    \includegraphics[width=.95\linewidth]{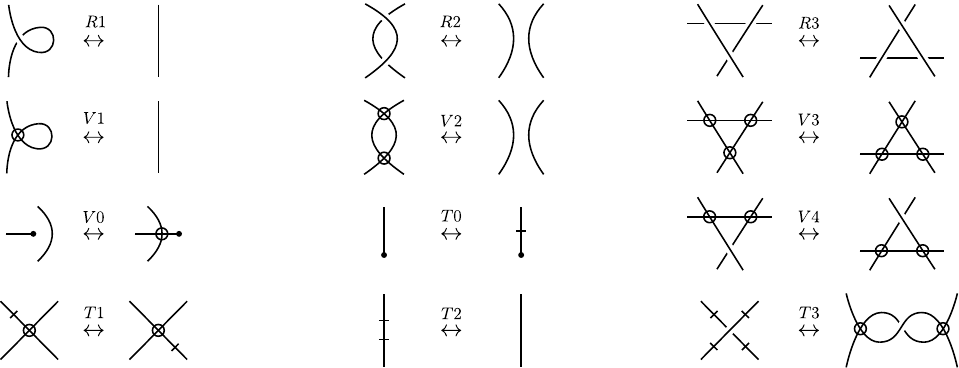}
     \caption{The twisted Reidemeister moves.}
    \label{fig:moves}
\end{figure}

To see that twisted knotoid diagrams modulo the equivalence generated by the moves in Figure \ref{fig:moves} indeed correspond with twisted knotoids, we adapt known results for twisted link diagrams from \cite{bourgoin2008} to knotoids. This involves  \emph{thickening} the diagrams and then capping them off with disks, for which we need the following definition.

\begin{definition}\label{def:thickening} 
A \emph{thickening} of a twisted knotoid diagram is formed by taking a small neighbourhood around the diagram's arcs to obtain an undecorated knotoid in a surface with boundary. For a twisted knotoid diagram this is realized as a thickening of the arcs into bands, with classical crossings being thickened into a `four-way junction' of bands, virtual crossings into crossings of bands, and bars into twists in the bands. See Figure \ref{fig:thickening}. For the virtual crossing, we note that it is not important which band goes over and which under, as the resulting surfaces with boundary are homeomorphic. 
To form a knotoid diagram on a compact surface $\Sigma$ from a twisted knotoid diagram, we sew disks into the boundary of its thickening. These disks are defined to be the \emph{regions} of the twisted knotoid diagram, and the attaching circles are called their \emph{boundaries}. See Figure \ref{fig:thickening} for an example.
\end{definition}


\begin{figure}[ht]
    \centering       \includegraphics[width=.65\linewidth]{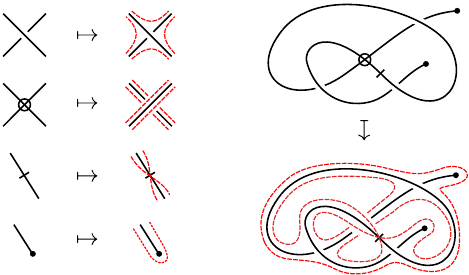}
    \caption{Rules for thickening a twisted knotoid diagram (left) and an example of the region boundaries of a twisted knotoid diagram with two regions (right).}
    \label{fig:thickening}
\end{figure}

It was shown in \cite{bourgoin2008} that, for twisted link diagrams, the moves in Figure \ref{fig:moves} are precisely the ones that, after thickening and capping off (and allowing for stabilization), generate equivalence of knot diagrams on compact surfaces $\Sigma$. To see that this result implies the analogous statement for knotoids, we note that an isotopy of knotoids on $\Sigma$ can be split into isotopies away from the endpoints and isotopies moving the endpoints. Isotopies away from the endpoints are known to be generated by the moves from Figure \ref{fig:moves} (excluding $V0$ and $T0$) due to the result for links. Isotopies near an endpoint may only move the endpoint within the region that contains it in $\Sigma$, and therefore correspond with a sequence of diagram isotopies and $V0$- and $T0$-moves.

It is further known that twisted link diagrams modulo twisted Reidemeister moves are equivalent to links in orientable thickenings $\Sigma \widetilde{\times} I$ of compact surfaces, modulo orientation-preserving homeomorphism and stable isotopy \cite{bourgoin2008}. Here the \textit{orientable thickening} $\Sigma \widetilde{\times} I$ of a compact surface $\Sigma$ is the orientation line bundle of $\Sigma$ with fiber $I$. Note that $\Sigma \widetilde{\times} I$ is always an orientable $3$-manifold with boundary, and that $\Sigma \widetilde{\times} I=\Sigma\times I$ if and only if $\Sigma$ is orientable. To obtain the analogous statement for twisted knotoids, we introduce `$H$-curves' in $\Sigma \widetilde{\times} I$:

\begin{definition}
Let $H$ denote the topological space given by a graph shaped like the letter `H', with its trivalent points labelled $h$ and $t$ (for `head' and `tail' respectively). We let $v_h,w_h$ denote the univalent vertices connected to $h$, and let $v_t,w_t$ denote those connected to $t$. Finally we let $e_h$ denote the path in $H$ between $v_h$ and $w_h$, and let $e_t$ denote that between $v_t$ and $w_t$.

An \emph{$H$-curve} is an embedding $\mathcal{H}:H\to \Sigma \widetilde{\times} I$ such that one of $\mathcal{H}(v_h)$ and $\mathcal{H}(w_h)$ lies on the 0-section of $\Sigma \widetilde{\times} I$ while the other lies on the 1-section, and similarly for $\mathcal{H}(v_t)$ and $\mathcal{H}(w_t)$. An $H$-curve is said to be \emph{simple} if $\mathcal{H}(e_h\cup e_t)$ is ambient isotopic to $\{p\times I\}\cup \{q\times I\}$ for some $p,q\in\Sigma$. The images $\mathcal{H}(e_h)$ and $\mathcal{H}(e_t)$ are also referred to as \textit{rails}.
\end{definition}

Like twisted knotoids, $H$-curves are considered up to orientation-preserving homeomorphisms of $\Sigma \widetilde{\times} I$ and \textit{stable equivalence}:

\begin{definition}
\cite{bourgoin2008} \textit{Stable equivalence} of $H$-curves in $\Sigma \widetilde{\times} I$ is the equivalence relation generated by ambient isotopies of $\Sigma \widetilde{\times} I$ that remain in the class of $H$-curves, and (de)stabilization moves. A \textit{destabilization move} on $\Sigma \widetilde{\times} I$ consists of cutting $\Sigma \widetilde{\times} I$ along the fiber of a closed loop in the zero-section that is disjoint from the $H$-curve, and capping off the resulting annular boundaries with thickened disks. A \textit{stabilization move} is the reverse of a destabilization move. For an example of a stabilization move see Figure \ref{fig:stabilization}.
\end{definition}


\begin{figure}[ht]
    \centering
    \includegraphics[clip, trim=0.2cm 8cm 21cm 10cm, width=0.8 \linewidth]{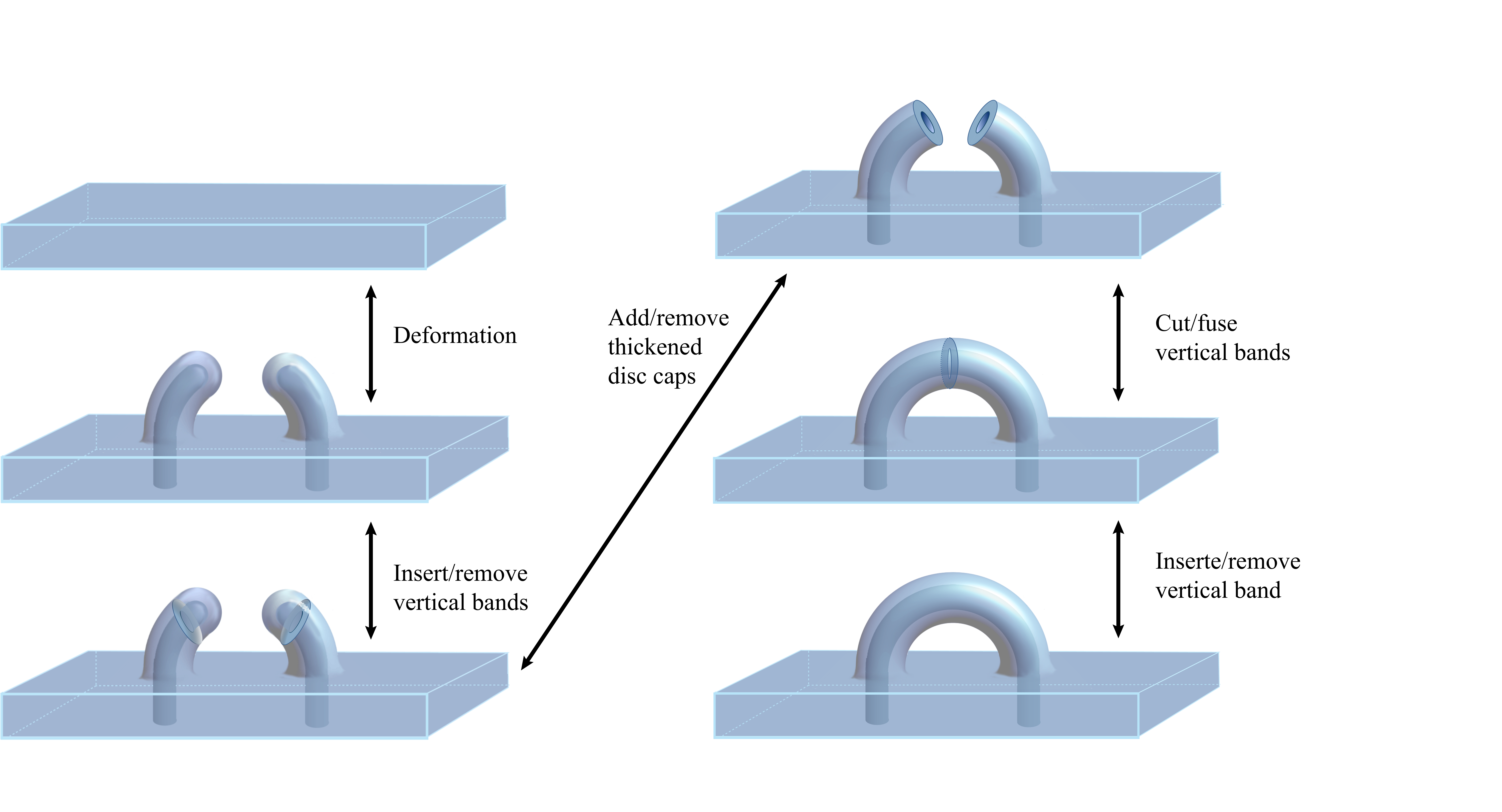}     
     \caption{A stabilization move on a thickened surface.}
    \label{fig:stabilization}
\end{figure}

\begin{proposition}\label{prop:realization}
There is a bijection between twisted knotoids and simple $H$-curves in thickened compact surfaces modulo stable equivalence.  
\end{proposition}

\begin{proof}
Let us call an $H$-curve in $\Sigma \widetilde{\times} I$ \emph{standard} if the images of $e_h$ and $e_t$ are of the form $\{p\}\times I$ for some $p\in \Sigma$. See Figure \ref{fig:H-curves}. By definition, every simple $H$-curve is equivalent to a standard $H$-curve. It is also easy to see that any equivalence of standard $H$-curves can be modified such that the $H$-curve at every step of the deformation remains standard. More formally, for any ambient isotopy $f_0: (\Sigma \widetilde{\times} I)\times I\to \Sigma \widetilde{\times} I$ interpolating between two standard $H$-curves, there exists a smooth map $F: (\Sigma \widetilde{\times} I)\times I\times I\to (\Sigma \widetilde{\times} I)\times I$ such that $F(-,-,-,t)$ is an ambient isotopy between the same $H$-curves for every $t\in I$, $F(-,-,-,0)=f_0$, and $F(-,-,s,1)$ is a standard $H$-curve for all $s\in I$. Hence we may restrict our attention to standard $H$-curves.

\begin{figure}[ht]
    \centering
    \includegraphics[width=.6\linewidth]{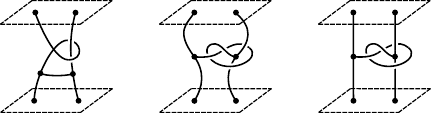}
     \caption{An $H$-curve that is neither simple nor standard (left), a simple non-standard $H$-curve (middle), and an equivalent standard $H$-curve (right).}
    \label{fig:H-curves}
\end{figure}


Let $\pi$ be the map sending a standard $H$-curve in $\Sigma \widetilde{\times} I$ to a knotoid in $\Sigma$ by projecting $\Sigma\times I$ onto the 0-section of $\Sigma \widetilde{\times} I$. This indeed yields a knotoid diagram, whose head and tail are respectively formed by the images of $e_h$ and $e_t$ in $\Sigma \widetilde{\times} I$, and whose crossing information is induced by the projection. Conversely, let $\iota$ be the map sending a knotoid with head $h$ and tail $t$ in $\Sigma$ to a standard $H$-curve in $\Sigma \widetilde{\times} I$ by drawing the knotoid in the $(1/2)$-section of $\Sigma \widetilde{\times} I$, perturbing it to have no singularities in accordance with the crossing information, and adding the lines $\{h\}\times I$ and $\{t\}\times I$ to obtain an $H$-curve.


Then it is clear that $\pi$ and $\iota$ are mutually inverse, so that they constitute the desired bijection provided they are well-defined. So it suffices to show that $\pi$ and $\iota$ factor through the equivalence relations on $H$-curves and knotoids respectively.

For well-definedness of $\iota$, note that an ambient isotopy of a knotoid diagram $K$ translates directly to an ambient isotopy of $\iota(K)$, so it suffices to consider well-definedness of $\iota$ under the twisted Reidemeister moves. Note that all of these moves, except for $V0$ and $T0$, are already known to correspond to local stable equivalence of $\iota(K)$ from the theories of virtual and twisted links \cite{kauffman2021,bourgoin2008}. So it is left to show that two knotoids related by $V0$- or $T0$-moves have stably equivalent images under $\iota$. This is shown in Figure \ref{fig:V0_and_T0}.

\begin{figure}[ht]
    \centering
    \includegraphics[clip, trim=0cm 0cm 36cm 0cm, width=0.8 \linewidth]{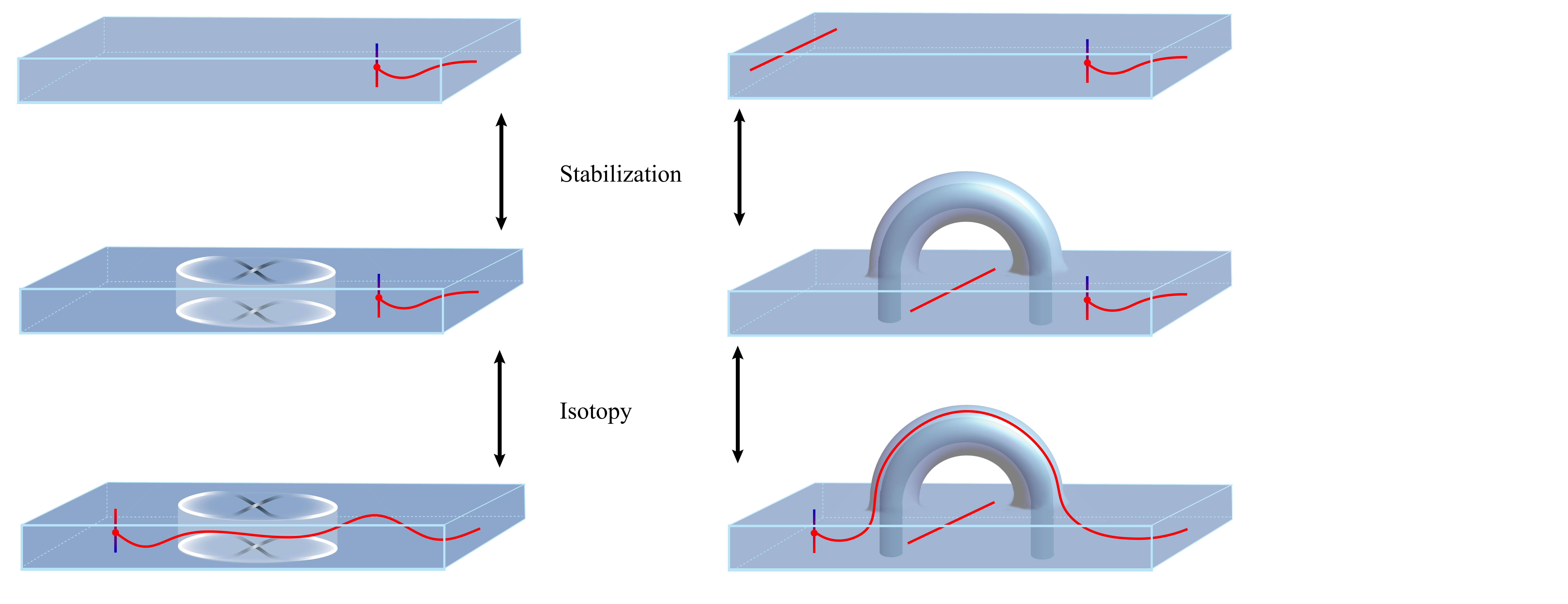}     
    \caption{Image of $T0$ (left-hand column) and $V0$ (right-hand column) moves under $\iota$. In both cases the image of the move constitutes a stable equivalence.}
    \label{fig:V0_and_T0}
\end{figure}



For well-definedness of $\pi$, note that stable equivalences of an $H$-curve $\mathcal{H}$ can be decomposed into a sequence of stable equivalences in neighbourhoods of $\mathcal{H}(e_h)$ or $\mathcal{H}(e_t)$, and stable equivalences away from them \cite{hudson1964combinatorial}. Well-definedness of $\pi$ under stable equivalences away from the rails is known from the theories of virtual and twisted links, so it suffices to consider stable equivalences near the rails. Note that stabilization moves have no effect on the diagram of $\pi(\mathcal{H})$, so we only need to consider isotopies near endpoint rails. By the assumption of standard-ness, such isotopies correspond directly with isotopies on $\Sigma$ near the endpoints of $\pi(\mathcal{H})$. In a diagram of $\pi(\mathcal{H})$, these isotopies on $\Sigma$ look like a sequence of ambient isotopies, $V0$ moves, and $T0$ moves.
\end{proof}

\begin{remark}
Note that Proposition \ref{prop:realization} subsumes the geometric realization of spherical knotoids by simple theta-curves in $S^3$ from \cite{turaev2012} once we restrict our attention to the case of the thickened sphere and disallow stabilization moves. Indeed, in this case we are just considering spherical knotoids as simple $H$-curves in $S^2\times I$, and such $H$-curves are in one-to-one correspondence with simple theta-curves with the correspondence being induced by the map $S^2\times I\to S^3$ that collapses $S^2\times \{0\}$ and $S^2\times \{1\}$ to points.
\end{remark}


\section{Arrow polynomials}\label{sec:arrow}

In this section we recall the arrow polynomial for spherical knotoids and the loop arrow polynomial for planar knotoids, and extend the former to twisted knotoids. To keep this discussion concise, we begin by discussing the generalities of arrow polynomial states, before defining the various arrow polynomials under consideration.

\subsection{State expansion and reductions}\label{subsec:states}

The basic premise of arrow polynomials is the oriented state expansion of crossings in an oriented diagram, depicted in Figure \ref{fig:expansion}. 
This expansion is analogous to the skein relation defining the Kauffman bracket polynomial \cite{kauffman1987}, but using the orientation 
to distinguish between smoothings that are consistent and inconsistent with the orientation. The disoriented smoothings are decorated with arrows at the points where the knotoid's orientation changes; see Figure \ref{fig:expansion}. Here the arrows are always placed to run counter-clockwise around where the crossing was. Note that this can be done consistently since all diagrams are on the plane or sphere, even though they may represent a knotoid on a non-orientable surface.

\begin{figure}[ht]
    \centering
    \includegraphics[width=.5\linewidth]{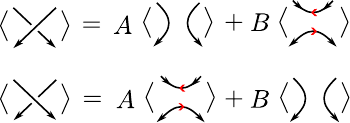}
    \caption{The oriented state expansion. The long black arrows indicate the original `tail-to-head' direction of the knotoid, while the short red arrows indicate point where the direction changes in the resulting states.}  
    \label{fig:expansion}
\end{figure}

\subsubsection{Spherical and virtual states}

For simplicity we begin by considering the expansion of spherical knotoids. Successive applications of the oriented state expansion reduce any spherical knotoid diagram to a $\mathbb{Z}[A,B]$-linear combination of crossing-free diagrams, known as \emph{states} of the knotoid. Each state is a disjoint union of closed loops and a single interval, known as \emph{state components}. These are some number of circular `\emph{loop components}' as well as a unique `\emph{arc component}', each of which is decorated with some number of arrows.

The states of a knotoid are further subject to reduction rules: two adjacent arrows pointing in the same direction may be cancelled. See Figure \ref{fig:reduction}. To produce the value of the arrow polynomials, states must be `fully reduced':

\begin{definition}
A state of an oriented state expansion is said to be \emph{fully reduced} if all its state components have the minimal number of arrows among diagrams that are equivalent modulo the reduction rules.
\end{definition}

\begin{figure}[ht]
    \centering
    \includegraphics[width=.17\linewidth]{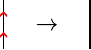}
    \caption{Reduction rule for arrows.}
    \label{fig:reduction}
\end{figure}

\begin{lemma}\label{lm:even_arrows}
The number of arrows on each state component of a state is even.
\end{lemma}

\begin{proof}
Any state consists of a tuple of loop components and a single arc component. Picking a point $x$ on an arbitrary loop component, consider the orientation of the state at $x$. Note that arrows on the state occur precisely at disoriented smoothings, where the orientation of the state is flipped. So running around the loop component from $x$ to itself, we must encounter an even number of arrows. Hence each loop component has an even number of arrows. Since the arrows are created in pairs with the smoothings, the total number of arrows on any state is even.  Thus, the single arc component must also have an even number of arrows. 
\end{proof}

We observe that the proof of Lemma \ref{lm:even_arrows} is completely elementary and therefore readily generalizes to twisted and planar knotoids. However,  the proof does rely on there being exactly one arc component and therefore does not generalize to linkoids \cite{gabrovvsek2023}. Indeed Lemma \ref{lm:even_arrows} does not hold for linkoids, as we shall see in Section \ref{sec:linkoids}.

Clearly a state is fully reduced if and only if all the arrows on it are alternating.
In combination with Lemma \ref{lm:even_arrows}, we conclude each state component in a fully reduced state is a loop with $2i$ alternating arrows, or an arc with $2i$ arrows. The arc components come in two distinct varieties, namely those whose first arrow (seen from the tail) points along the orientation at the arc's tail, and those whose first arrow does not. We label these reduced state components by variables $K_i$, $\Lambda_i$, and $\Lambda'_i$ respectively.


\begin{remark}
For virtual knotoids, by which we mean twisted knotoid diagrams with no bars considered up to the $R$- and $V$-moves only (recall Figure \ref{fig:moves}), the story is much the same. In the presence of virtual crossings we call a diagram crossing-free if it contains no classical crossings. Successive applications of the oriented state expansion reduce any virtual knotoid to a crossing-free diagram, consisting of a set of state components with only virtual crossings. Allowing arrows to pass virtual crossings, any diagram with only virtual crossings is trivial, i.e.~equivalent to a diagram without virtual crossings. From this we conclude that also in the virtual case, each reduced state is equivalent to a disjoint union of copies of $K_i$, $\Lambda_i$, and $\Lambda'_i$.
\end{remark}

\subsubsection{Twisted states}

In the case of twisted knotoids we must account for the presence of bars. First we allow $T0$ and $T2$ moves on state diagrams. An additional reduction rule for states decorated both with arrows and with bars must be introduced \cite{deng2022}; see the right side of Figure \ref{fig:twisted_reduction}. A state containing both arrows and bars is said to be fully reduced if it has a minimal number of bars and of arrows.

\begin{figure}[ht]
    \centering
    \includegraphics[width=.5\linewidth]{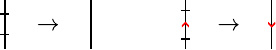}
    \caption{Reduction rules for bars.}
    \label{fig:twisted_reduction}
\end{figure}

\begin{lemma}\label{lm:bar_passing}
The reduction rules imply that we may pass a bar over an arrow, changing the arrow's direction in the process. See Figure \ref{fig:bar_passing}.
\end{lemma}
\begin{proof}
This follows by applying a $T2$ move. See Figure \ref{fig:bar_passing}.
\end{proof}

\begin{figure}[ht]
    \centering
    \includegraphics[width=.4\linewidth]{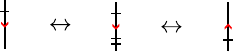}
    \caption{Moving a bar over an arrow.}
    \label{fig:bar_passing}
\end{figure}



\begin{lemma}
Let $C$ be a reduced state component in a twisted knotoid state. Then if $C$ has arrows it is equivalent to one of $\{K_i,\Lambda_i\}_{i\in\mathbb{N}}$, or else $C$ is equivalent to a loop with a single bar on it, which we label by a variable $K_{1/2}$. Here $K_i,\Lambda_i$ denote the same reduced state components as for the classical and virtual cases.
\end{lemma}

\begin{proof}
Suppose $C$ is an arc component of a state. Then to put $C$ in fully reduced form, move all the bars on $C$ to its tail using Lemma \ref{lm:bar_passing}. Next apply reduction moves to cancel any non-alternating pairs of arrows, and apply $T0$ moves to remove all the bars. Then the result is either $\Lambda_i$ or $\Lambda'_i$ which are both clearly fully reduced. Finally we note that $\Lambda_i$ and $\Lambda'_i$ are equivalent in the twisted setting, so that $C$ is equivalent to $\Lambda_i$ without loss of generality. Indeed, to obtain $\Lambda'_i$ from $\Lambda_i$, apply a $T0$ move to the tail of the arc component. Then pass this bar over all the arrows on the component. The result is a copy of $\Lambda'_i$ with a bar at its head. Applying another $T0$ move then yields $\Lambda'_i$.

Next, if $C$ is a loop component suppose $C$ contains an even number of bars. Then $C$ can be reduced by moving all the bars on it to a neighbourhood of some point on $C$ using Lemma \ref{lm:bar_passing}, and cancelling them all using $T2$ moves. Afterwards we apply reduction moves cancelling non-alternating arrows to obtain a copy of $K_i$, which is clearly fully reduced.

Finally, if $C$ is a loop component with an odd number of bars, it can again be reduced by collecting all the bars around a single point, removing them in pairs using $T2$ moves, and finally removing any non-alternating arrows. The result is a loop with alternating arrows and a single bar. Moving the bar over one of its adjacent arrows flips the arrow, creating a pair of non-alternating arrows which we cancel. Doing so repeatedly cancels all arrows on the loop, leaving us with a copy of $K_{1/2}$ which is again clearly fully reduced.
\end{proof}

\subsubsection{Planar states}

Finally, we discuss the states coming from planar knotoids. Due to the puncture present in their diagrams (recall Remark \ref{rk:punctures}) these have additional structure. In another sense, however, states of planar knotoid diagrams have less structure. Namely these diagrams have only classical crossings which implies the states will have no components $K_i$:

\begin{lemma}\label{lm:classical_arrows}
Let $K$ be a classical knotoid diagram. Then all the loop components in a reduced state of $K$ are free of arrows.
\end{lemma}

\begin{proof}
The argument from \cite[Sec.~8]{kauffman2012} for the analogous statement for classical links generalizes directly. Indeed: this argument shows that any reduced loop with arrows on it must have another nontrivial reduced state component in its interior region, and one in its exterior region. Here the interior/exterior regions are those regions of the plane associated to the loop by the Jordan curve theorem. Suppose we are given a nontrivial reduced loop component. Then only one of these regions can contain the arc component, meaning the other must contain infinitely many loops. This contradicts the assumption that $K$ has finitely many crossings.
\end{proof}

Clearly the proof of Lemma \ref{lm:classical_arrows} makes critical use of the assumption that our component is a loop, as it invokes the Jordan curve theorem. This assumption is necessary, meaning the arc component can have arrows on its fully reduced form.

In conclusion the only arrows on a fully reduced planar knotoid state lie on its arc component. The presence of a puncture allows us to attach more information to this arc component: due to the puncture, for planar knotoids it no longer holds that any state is equal to a disjoint union of its state components. Namely, if a loop component of the state encircles the arc component but not the puncture then there is no isotopy moving the loop so that it has an empty interior, therefore it cannot be a factor of a disjoint union. (Note that in the classical setting of planar knotoids we don't have $V0$ moves at our disposal.)
Thus to classify the disjoint union factors of a planar knotoid state, we must keep track of the number of loops separating the arc component from the puncture.

\begin{lemma}
Let $K$ be a planar knotoid, and let $S$ be a state of $K$. Then the fully reduced form of $S$ is a disjoint union of loops with no arrows and a copy of $\Lambda_i$ or $\Lambda'_i$ encircled by $\ell$ loops, with the puncture of the diagram lying outside the outer loop that encircles the arc component. We denote this last factor of the disjoint union by $\left( \Lambda^{(\prime)}_i \right)^\ell$.
\end{lemma}

\begin{proof}
Since $K$ is planar and hence classical, $S$ is free of self-intersections.  Let $C$ be a loop component of $S$. Then by the Jordan curve theorem $C$ has an interior and an exterior. If one of these contains neither the arc component of $S$ nor the puncture, then $S=S'\sqcup C$ for some smaller state diagram $S'$. Otherwise $C$ is one of the $\ell$ loops encircling the arc component. Since $S$ is free of self-intersections, all such loops must be equivalent to a nested set of loops with the arc at its center and the puncture in its exterior.
\end{proof}


\subsection{Arrow polynomials}

Having taken the time to discuss the oriented state expansion and reduction rules at some length, we can define the arrow polynomial for twisted knotoids with little effort. We begin with recalling the arrow polynomial for virtual knotoids \cite{gugumcu2017}:

\begin{definition}\label{def:arrow}
Let $K$ be a virtual knotoid diagram. The \emph{arrow bracket} of $K$, denoted $\langle K\rangle$ is defined to be the unique polynomial such that:
\begin{enumerate}
    \item $\langle K\rangle$ satisfies the oriented state expansion formulas from Figure \ref{fig:expansion}.
    \item On states, $\langle K\rangle$ distributes multiplicatively over disjoint unions of state components.
    \item $\langle K_0\rangle = d$ and $\langle \Lambda_0 \rangle = 1$.
    \item $\langle K_i\rangle=dK_i$ and $\langle \Lambda^{(\prime)}_i \rangle=\Lambda^{(\prime)}_i$, where we introduce the symbols $K_i$, $\Lambda_i$, and $\Lambda^{\prime}_i$ as commuting variables in the value of $\langle K\rangle$.
\end{enumerate}
Thus $\langle K\rangle$ takes values in the ring of polynomials over $A,B,d$, and countably many commuting variables indexed by the fully reduced state components for virtual knotoids, or in other words $\langle K\rangle \in \mathbb{Z}[A,B,d, \{K_i, \Lambda_i, \Lambda'_i\}_{i\in \mathbb{Z}_{\geq 0}} ]$.
\end{definition}

We note that each summand of $\langle K\rangle$ is linear in exactly one of $\Lambda_i$ or $\Lambda_i^\prime$ for $i\geq 0$, and constant in the other. Higher powers of $\Lambda_i$ and $\Lambda_i^{(\prime)}$ are attained by the generalization of $\langle K\rangle$ to \textit{linkoids}; see Definition \ref{def:linkoid_poly}.


\begin{definition}
For $K$ a virtual knotoid diagram, the \emph{arrow polynomial} $\langle K\rangle_A$ of $K$ is defined to be $\langle K\rangle$ evaluated at $B=A^{-1}$ and $d=(-A^2-A^{-2})$. As such $\langle K\rangle_A\in \mathbb{Z}[A^{\pm1}, \{K_i, \Lambda_i, \Lambda'_i\}_{i\in \mathbb{Z}_{\geq 0}} ]$.
\end{definition}



The arrow polynomial is itself not an invariant of virtual knotoids, but of \emph{framed} virtual  knotoids. These are virtual knotoid diagrams on $S^2$ considered up to all the relevant moves from Figure \ref{fig:moves} except for $R1$; see \cite{moltmaker2022,moltmaker2023} for a detailed account of framed classical knotoids. However, the effect of an $R1$ move on $\langle K\rangle_A$ is just multiplication by a factor $-A^{\pm 3}$, depending on the sign of the crossing involved in the move. As such $\langle K\rangle_A$ is easily `unframed' using the writhe $\text{wr}(K)$:

\begin{proposition}\label{prop:invariance}
The normalized arrow polynomial $(-A^3)^{-\text{wr}(K)} \langle K\rangle_A$ is an invariant of virtual knotoids.
\end{proposition}

For a proof of this proposition, see \cite{gugumcu2017}. As a corollary to Lemma \ref{lm:classical_arrows} recall that classical states will have no components $K_i$:

\begin{corollary}
For spherical knotoids, $\langle K\rangle_A\in \mathbb{Z}[A^{\pm1}, \{\Lambda_i\}_{i\in \mathbb{N}}, \{\Lambda'_i\}_{i\in \mathbb{N}} ]$.
\end{corollary}

In fact we have the following closed expression for $\langle K\rangle_A$ on spherical knotoids:
\[
    \langle K\rangle_A = \sum_{s\in S(K)} A^{\sigma(s)} (-A^2-A^{-2})^{\lvert s \rvert -1} \langle s \rangle.
\]
Here $S(K)$ is the set of states of $K$, $\sigma(s)\in\mathbb{Z}$ is the number of $A$-smoothings in $s$ minus the number of $B$-smoothings in $s$ (refer to Figure \ref{fig:expansion}), and $\lvert s\rvert$ is the number of components in $s$. Finally $\langle s \rangle\in \{1,\Lambda_i,\Lambda'_i\}_{i\in\mathbb{N}}$ is the variable corresponding to the reduced arc component in $s$.

\begin{example}
Let $K$ be the spherical knotoid depicted in Figure \ref{fig:spherical_example}, which also depicts the states of $K$ and their weights. Then from these weights we find that
\[
    \langle K\rangle_A = \left(- A^3 + A^{-1} -A^{-5}\right) + \left( A - A^5 \right)\Lambda_1.
\]


\begin{figure}[ht]
    \centering
    \includegraphics[width=.85\linewidth]{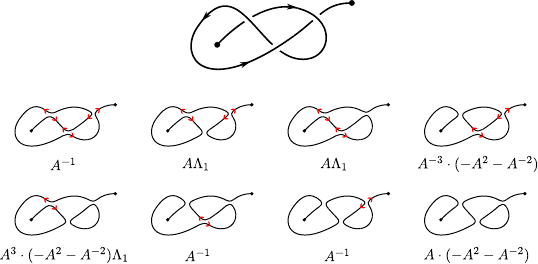}
    \caption{Example spherical knotoid $K$ and its states.}
    \label{fig:spherical_example}
\end{figure}
\end{example}

Next we extend the arrow polynomial to twisted knotoids:

\begin{definition}
Let $K$ be a twisted knotoid diagram. The \emph{twisted arrow bracket} $\langle K\rangle^t$ of $K$ is the unique polynomial satisfying the same rules as in Definition \ref{def:arrow}, except that in rule 4 we further introduce a variable $K_{1/2}$, declare that $\langle K_{1/2} \rangle^t = K_{1/2}$, and set $\Lambda'_i=\Lambda_i$. The \emph{twisted arrow polynomial} $\langle K\rangle^t_A$ of $K$ is then defined as the $\mathbb{Z}[A^{\pm1}, K_{1/2}, \{K_i, \Lambda_i\}_{i\in \mathbb{Z}_{\geq 0}} ]$-valued polynomial obtained from $\langle K\rangle^t$ by the substitution $B=A^{-1}$, $d=(-A^2-A^{-2})$.
\end{definition}

Note in particular that a virtual knotoid diagram, when seen as a twisted knotoid diagram, will have a twisted arrow polynomial different from its arrow polynomial: they can be made equal by putting $\Lambda'_i=\Lambda_i$. This is our reason for reserving separate notation $\langle K\rangle^t_A$ for the twisted arrow polynomial of virtual knotoids. Again, the twisted arrow polynomial is only an invariant of \emph{framed} twisted knotoids, but is easily unframed:

\begin{proposition}
The normalized twisted arrow polynomial $(-A^3)^{-\text{wr}(K)} \langle K\rangle^t_A$ is an invariant of twisted knotoids.
\end{proposition}

\begin{proof}
Except for $V0$ and $T0$, invariance under the moves from Figure \ref{fig:moves} follows directly from the proof of the analogous statement for virtual links; see \cite{deng2022}. Invariance under $V0$ and $T0$ is trivial since these moves are also allowed for state diagrams.
\end{proof}

\begin{example}\label{ex:twisted}
Let $K$ be the twisted knotoid depicted in Figure \ref{fig:twisted_example}, which also depicts the states of $K$ and their weights. Then we find that
\[
    \langle K\rangle_A^t = \left(A^2 + 1 + A^{-2}\right) + \left( -A^2 - A^{-2} \right)\Lambda_1 K_1.
\]

\begin{figure}[ht]
    \centering
    \includegraphics[width=.5\linewidth]{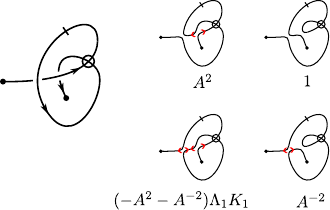}
    \caption{Example twisted knotoid $K$ and its states.}
    \label{fig:twisted_example}
\end{figure}
\end{example}

\begin{remark}\label{rk:arrow_mult}
    There is a natural operation of concatenation for knotoids and virtual 
    knotoids. It is defined for diagrams and consists in attaching the
    head of the first knotoid to the tail of the second one, so that the
    tail of the new knotoid is the tail of the first of the two knotoids,
    while its head is that of the second one.
    Obviously, this operation is well-defined: concatenation of equivalent 
    knotoid diagrams produces equivalent knotoid diagrams. 
    It is, however, not a commutative operation.
    The arrow polynomial is multiplicative with respect to concatenation,
    if we impose $\Lambda_i\cdot \Lambda_j = \Lambda_{i+j}$ and $\Lambda^\prime_i\cdot \Lambda^\prime_j = \Lambda^\prime_{i+j}$, and for $i>j$ impose that $\Lambda^\prime_i\cdot \Lambda_j = \Lambda^\prime_{i-j}$ and $\Lambda_i\cdot \Lambda^\prime_j = \Lambda_{i-j}$.
\end{remark}

Finally we define the arrow polynomial for planar knotoids, which is called the `loop' arrow polynomial \cite{goundaroulis2017}. This definition is nearly identical to that for virtual knotoids, except that planar states decompose very differently from virtual states. Hence the loop arrow polynomial takes values in a polynomial ring with different variables, again necessitating a separate definition and notation.

\begin{definition}
Let $K$ be a planar knotoid diagram. The \emph{loop arrow bracket} of $K$, denoted $\langle K\rangle^\ell$, is the unique polynomial satisfying the same rules as in Definition \ref{def:arrow}, carefully noting the meaning of `disjoint union' for the planar setting in rule 2. The \emph{loop arrow polynomial} $\langle K\rangle^\ell$ of $K$ is obtained by the substitution $B=A^{-1}$, $d=(-A^2-A^{-2})$ as before.
\end{definition}

As a corollary to the proof of Proposition \ref{prop:invariance}, the normalized loop arrow polynomial $(-A^3)^{-\text{wr}(K)} \langle K\rangle^\ell_A$ is an invariant of planar knotoids.

\begin{example}
Let $K$ be the planar knotoid depicted in Figure \ref{fig:twisted_example}, in which the exterior region is made explicit by a point at infinity (recall Remark \ref{rk:punctures}). From the states and weights of $K$ in Figure \ref{fig:planar_example} we find
\[
    \langle K\rangle_A^\ell = A^{-3} + (A^{-1} - A^3) \Lambda_1 + A^{-1} (\Lambda_0)^1 + 2A (\Lambda_1)^1 + A^3(\Lambda_1)^2.
\]

\begin{figure}[ht]
    \centering
    \includegraphics[width=.75\linewidth]{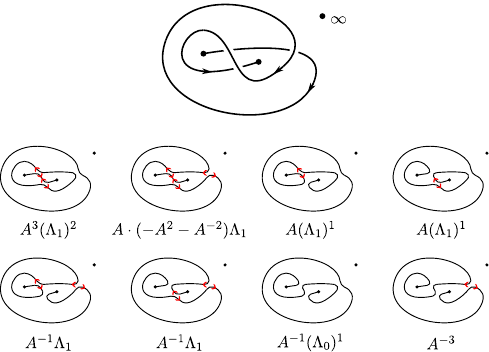}
    \caption{Example planar knotoid $K$ and its states.}
    \label{fig:planar_example}
\end{figure}
\end{example}


\section{Marked ribbon graphs and Bollob\'{a}s-Riordan polynomials}\label{sec:BR}

In this section we introduce (various augmentations of) the Bollob\'{a}s-Riordan polynomial. This is a polynomial invariant of ribbon graphs. As such we first introduce ribbon graphs, with various decorations.


\subsection{Marked ribbon graphs}

\begin{definition}
A \emph{ribbon graph} $G=(V(G),E(G))$ is a surface with boundary consisting of a union of two sets of disks, a set $V(G)$ of \emph{vertices} and a set $E(G)$ of \emph{edges}, such that:
\begin{enumerate}
    \item The vertices and edges intersect in disjoint line segments.
    \item Each such line segment lies on exactly one vertex and exactly one edge.
    \item Each edge contains exactly two such intersections.
\end{enumerate}
\end{definition}

Ribbon graphs are naturally represented by drawings, known as ribbon graph diagrams, in the same way that graphs are: we draw vertices as circular disks connected by edges drawn as long ribbons, whose short sides intersect the disks along their boundary. See Figure \ref{fig:ribbon_graph} for an example. In these diagrams, if any edges cross each other we don't make keep track of any over/under-crossing information, as both choices for this information yield the same surface (up to a homeomorphism that induces the identity isomorphism of the underlying abstract graph) and hence represent the same ribbon graph. Further, if the edge make some number of half-twists then two adjacent half-twists cancel since putting a rotation of $2\pi$ into a ribbon does not change the resulting surface (again, up to homeomorphism). Thus without loss of generality each edge in a ribbon graph diagram contains at most one half-twist.

\begin{figure}[ht]
    \centering
    \includegraphics[width=.25\linewidth]{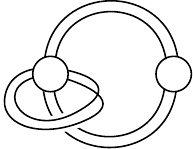}
    \caption{A ribbon graph diagram.}
    \label{fig:ribbon_graph}
\end{figure}

\begin{remark}\label{rk:ribbon_equivalences}
Ribbon graphs are known to be equivalent to cellularly embedded graphs in compact surfaces, with the equivalence being given by noting that the boundary of a ribbon graph $G$ is a set of circles and that sewing disks onto each of these boundary circles casts $G$ as being cellularly embedded in the resulting compact surface. We also note that ribbon graphs are equivalent to \emph{signed rotation systems}, which consist of an abstract graph, together with a cyclic ordering of the half-edges at every vertex as well as a sign `$+$' or `$-$' associated to every edge. These rotation systems are considered up to \emph{local switches}, which mean reversing the cyclic order at some vertex, and toggling the sign of every edge incident to that vertex. The bijection from ribbon graphs to rotation systems is then given by sending a ribbon graph diagram to its underlying graph, with the cyclic ordering being that of the edge ribbons on the original vertex discs, and the sign of an edge being determined by whether or not the edge makes a half-twist. See \cite{ellis2013} for more details.
\end{remark}

For the purposes of Thistlethwaite's theorem, one associates (edge-labelled) ribbon graphs to link diagrams roughly as follows: choosing a state in the Kauffman state sum for the bracket polynomial of a link diagram (recall \cite{kauffman1987}, or see Section \ref{sec:arrow}),
let the vertices be the resulting disjoint cycles the diagram falls into, and place edges bridging each of the smoothed crossings in the state. The edges are further labeled with a `$+$' or a `$-$', in accordance with the power of A that is incurred by the Kauffman bracket from the choice of smoothing at that edge's site. The original link diagram can then be recovered from this edge-labelled ribbon graph.


This rough outline of the construction serves as a motivating for defining `marked' ribbon graphs. A more detailed account will follow (for the case of knotoids) where it is relevant, in Section \ref{sec:Thistlethwaite}. For now, suppose we wish to apply the same construction to a knotoid diagram. Doing so, the object we obtain is almost a ribbon graph, except that a Kauffman state of a knotoid diagram consists of disjoint cycles as well as a single arc component, which we also wish to interpret as a ribbon graph vertex. In light of Remark \ref{rk:ribbon_equivalences}, we can treat this arc component of the state as a ribbon graph vertex but with a \emph{linear order} on its incident half-edges, rather than a cyclic order. This naturally leads to the definition of a `marked' ribbon graph:

\begin{definition}
A \emph{marked ribbon graph} is a ribbon graph $G$, some of whose vertex disks have been decorated by a single \emph{marking} on the vertex's boundary away from intersections with the edges of $G$. These markings 
serve as a `starting point' for the cyclic order of half-edges on the vertex, thereby linearizing it. Marked vertices are equipped with a choice of linearization, i.e. an orientation of their vertex-disc boundaries.
\end{definition}

In most of the remainder of this paper we will encounter marked ribbon graphs with just one marked vertex. In Section \ref{sec:linkoids} we briefly treat the case of marked ribbon graphs with any number of marked vertices.

A marked ribbon graph diagram is just a ribbon graph diagram, with the markings drawn on it. To evoke the application of marked ribbon graphs that we have in mind here, we shall draw the markings as two endpoints connected with a dotted line completing the vertex disk, leaving the impression that this vertex disk is in some sense actually an arc. See Figure \ref{fig:marked_ribbon_graphs} for an example.


The marked ribbon graphs we will be concerned with in this paper may further be decorated in various ways: the edges will be \emph{signed}, i.e.~labelled with either a `$+$' or a `$-$', and the vertices and edges will be decorated with \emph{arrows} and \emph{bars}. These arrows can lie anywhere on the boundaries of the vertices and edges, including on their intersection arcs. 
The bars can only lie on the boundaries of vertices, and must be placed away from intersection arcs with edges. Bars are drawn as small line segments transversal to the vertices' boundaries. See Figure \ref{fig:marked_ribbon_graphs} for an example marked ribbon graph with all possible decorations.

\begin{figure}[ht]
    \centering
    \includegraphics[width=.6\linewidth]{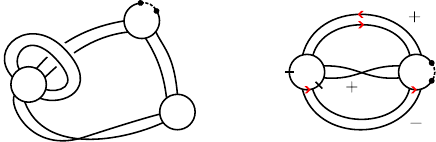}
    \caption{A marked ribbon graph diagram (left) and another decorated with signs, arrows, and bars (right). 
    }
    \label{fig:marked_ribbon_graphs}
\end{figure}

\begin{remark}
For the purposes of Thistlethwaite's theorem, one obtains a marked ribbon graph from a knotoid diagram as described above using a Kauffman state. The components of such a state are a set of loops and a single interval. The loops of the state again form its vertices, and the interval forms a marked vertex by fusing its endpoints into a marking on a vertex. This last step to obtain a marked ribbon graph may seem unnecessary: we could also have worked with the ribbon graph diagram in which the linearized vertex is represented by an arc, instead of a marked vertex. Representing marked vertices as arcs in this way results in what one might call a `ribbonoid graph'. Since the data included in a marked ribbon graph is purely combinatorial, these formulations are entirely equivalent. In this article we choose to work solely with marked ribbon graphs for clarity, as these more closely resemble the familiar framework of ribbon graphs.
\end{remark}

A fundamental operation for ribbon graphs is \textit{partial duality}. This operation generalizes the standard duality for embedded graphs to allow one to take the dual of a subset of the graph's edges. The standard dual of a graph is then equal to its partial dual with respect to the entire edge-set. Here we define partial duality for marked ribbon graphs.

\begin{definition}
Let $G$ be a marked ribbon graph, potentially decorated with arrows, and $e$ an edge of $G$. Then the \textit{partial dual} $G^e$ of $G$ with respect to $e$ is obtained by swapping which boundary arcs of $e$ are attaching arcs for vertices, and which are not. More precisely, the vertex-attaching arcs of $e$ are removed, and the remaining arcs of $e$ are interpreted as new vertex-arcs. Then, new edge-arcs are placed parallel to the removed attaching arcs, forming an edge-disk that attaches to the new vertex-arcs. See Figure \ref{fig:partial_dual}, where the edge $e$ in $G^e$ has been pushed out of the plane for clarity. Any arrow decorations or vertex markings on the ribbon graph are left in place during this process, with those present on the removed attaching arcs being placed on the new edge-arcs; see Figure \ref{fig:partial_dual}.

If $D\subseteq E(G)$ is a subset of edges of $G$ then the \textit{partial dual} $G^D$ of $G$ with respect to $D$ is formed by carrying out the same local replacement for each $e\in D$.
\end{definition}

\begin{figure}[ht]
    \centering
    \includegraphics[width=.55\linewidth]{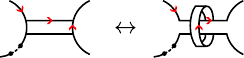}
    \caption{Partial duality with respect to an edge marked with arrows. The vertices may be marked, e.g.~as depicted, or unmarked.}
    \label{fig:partial_dual}
\end{figure}

As terminology suggests partial duality is a self-dual operation, in the sense that $(G^D)^D=G$ for all $D\subseteq E(G)$, so that the two-sided arrow in Figure \ref{fig:partial_dual} is indeed justified. Moreover the partial dual of a ribbon graph with respect to $E(G)$ yields the usual dual of $G$, i.e. $G^{E(G)}=G^*$.

\subsection{Bollob\'{a}s-Riordan polynomials}

In this section we restrict our attention to marked ribbon graphs with only a single marked vertex, as these are the graphs arising from knotoids; see Section \ref{sec:Thistlethwaite}. For a discussion of the case of marked ribbon graphs with several marked vertices, see Section \ref{sec:linkoids}.

The multivariate Bollob\'{a}s-Riordan polynomial for (edge-weighted) 
ribbon graphs can be defined as follows:
\[
    Z_G(a,\mathbf{b},c) = \sum_{F\subseteq E(G)} a^{k(F)} \left( \prod_{e\in F} b_e \right) c^{\text{bc}(F)}.
\]
Here the sum ranges over all graphs $F$ given by keeping only a subset of the edges in $G$, i.e.~the spanning subgraphs of $G$. We also refer to such subgraphs as the \emph{states} of $G$. Furthermore, $k(F)$ is the number of connected components of $F$, $\text{bc}(F)$ is the number of boundary components of $F$, and $\mathbf{b}$ is an array of variables $b_e$ indexed by the edges of $G$.



We extend the multivariate Bollob\'{a}s-Riordan polynomial to decorated marked ribbon graphs by further keeping track of the reduced state information on the boundary components of spanning subgraphs of $G$. We do so in three different ways, corresponding with the three paradigms for which we have separate arrow polynomials: virtual, twisted, and planar knotoids.

\begin{definition}
The \emph{arrow Bollob\'{a}s-Riordan polynomial} for marked ribbon graphs decorated with arrows is defined by
\[
    R_G(a,\mathbf{b},c) = \sum_{F\subseteq E(G)} a^{k(F)} \left( \prod_{e\in F} b_e \right) c^{\text{bc}(F)} \left( \prod_{f\in\partial^c(F)} K_{i(f)/2} \right) \Lambda_i^{(\prime)}.
\]
Here $\partial^c(F)$ denotes the set of circular boundary components of $F$, i.e.~the set of all unmarked boundary components, and $i(f)$ denotes the number of arrows on a component $f\in \partial^c(F)$. The variables $K_i,\Lambda_i,\Lambda_i'$ denote the reduced state components of the boundary of $F$, subject to the same reduction rule from Figure \ref{fig:reduction}. We set $K_0=\Lambda_0^{(\prime)}=1$.
\end{definition}

\begin{definition}
The \emph{twisted Bollob\'{a}s-Riordan polynomial} for marked ribbon graphs decorated with bars and arrows is defined by
\[
    R^t_G(a,\mathbf{b},c) = \sum_{F\subseteq E(G)} a^{k(F)} \left( \prod_{e\in F} b_e \right) c^{\text{bc}(F)} \left( \prod_{f\in\partial^c(F)} K_{i(f)/2} \right) \Lambda_i.
\]
Here $\partial^c(F)$, $i(f)$, and the variables $K_i,\Lambda_i$ are as before, now subject to the reduction rules from Figures \ref{fig:reduction} and \ref{fig:twisted_reduction}.
\end{definition}

\begin{definition}
A \textit{punctured marked ribbon graph} is a marked ribbon graph one of whose edge-discs, vertex-discs, or faces (after capping off boundaries) contains a marked point known as the \textit{puncture}. The \emph{loop arrow Bollob\'{a}s-Riordan polynomial} for punctured marked ribbon graphs decorated with arrows is defined by
\[
    R^\ell_G(a,\mathbf{b},c) = \sum_{F\subseteq E(G)} a^{k(F)} \left( \prod_{e\in F} b_e \right) c^{\text{bc}(F)-\ell} \left( \Lambda_i^{(\prime)} \right)^\ell.
\]
Here the variables $(\Lambda_i)^\ell,(\Lambda'_i)^\ell$ denote the reduced state components of the boundary components of $F$ as before.
\end{definition}



\begin{example}\label{ex:graph_poly}
Let $G$ be the marked ribbon graph with arrows depicted in Figure \ref{fig:graph_poly}. 
We will compute $R_G(a,\mathbf{b},c)$. 
The subgraphs corresponding to each subset $F\subseteq \{e_1,e_2\}$ to $R_G$, as well as their contributions, are also given in Figure \ref{fig:graph_poly}. From this, we find that
\[
    R_G(a,\mathbf{b},c) = ab_{e_2}c + (ac + ab_{e_1}c^2 + ab_{e_1}b_{e_2}c )\Lambda_1.
\]

\begin{figure}[ht]
    \centering
    \includegraphics[width=.55\linewidth]{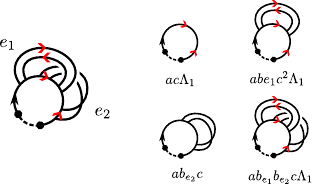}
    \caption{An example marked ribbon graph $G$ with arrow decorations, its subgraphs, and their contributions.}
    \label{fig:graph_poly}
\end{figure}
\end{example}

Next we formulate the contraction-deletion relation for the {arrow Bollob\'{a}s-Riordan polynomial} for marked ribbon graphs decorated with arrows. Here deletion and contraction of edges of marked ribbon graphs are defined as usual for ribbon graphs \cite{bradford2012arrow}, with the rule for the arrow decorations that they are kept in place, except if they land in the interior of a face or vertex after deletion or contraction in which case they are removed. To be precise, this happens when an arrow lies on the boundary of a deleted edge but not on an attaching vertex, or if it lies on the attaching arc of a contracted edge.



\begin{proposition}\label{prop:del_cont}
For $G$ a marked ribbon graph decorated with arrows, the arrow Bollob\'{a}s-Riordan polynomial $R_G(a,\mathbf{b},c)$ possesses the following properties with respect to multiplication, deletion, and contraction.
\begin{align*}
R_{G_1\sqcup G_2} &= R_{G_1}\cdot R_{G_2} ;\\
R_G(a,\mathbf{b},c) &= 
\begin{cases}
R_{G-e}(a,\mathbf{b},c) + b_e R_{G/e}(a,\mathbf{b},c) & \text{if $e$ is not an orientable loop,}\\
R_{G-e}(a,\mathbf{b},c) + (b_e/a)R_{G/e}(a,\mathbf{b},c) & \text{if $e$ is a trivial orientable loop.}
\end{cases}
\end{align*}
\end{proposition}

\begin{proof}




The first property for $R_{G_1\sqcup G_2}$ is immediate. 
The proof of the contraction-deletion properties follows the standard procedure. For $F\subseteq E(G)$ we denote
\[
    W(F) = a^{k(F)} \left( \prod_{e\in F} b_e \right) c^{\text{bc}(F)} \left( \prod_{f\in\partial^c(F)} K_{i(f)/2} \right) \Lambda_i^{(\prime)}.
\]
One can split the set of edge subsets $F\subseteq E(G)$ into two types according to the property $e\in F$ or $e\not\in F$. The subgraphs of the first and second type may be regarded as edge subsets of $G/e$ and $G-e$, respectively. For an edge $e$ which is not an orientable loop, the exponents of variables $a$ and $c$ will be preserved when we consider $F$ as a spanning subgraph of $G$, or of either $G/e$ or $G-e$. That is, if $e$ is not an orientable loop, then 
\begin{align*}
    R_G(a,\mathbf{b},c) &= 
    \sum_{F\subseteq E(G)} W(F) \\
    &= \sum_{e\notin F_{1}, F_{1}\subseteq E(G)} W(F_1) + \sum_{e\in F_{2}, F_{2}\subseteq E(G)} W(F_2)\\
    &= \sum_{F_{1}\subseteq E(G-e)} W(F_1) + b_{e}\sum_{F_{3}\subseteq E(G/e)} W(F_3)\\
    &= R_{G-e}(a,\mathbf{b},c) + b_{e} R_{G/e}(a,\mathbf{b},c).
\end{align*}
Hence we obtain the first contraction-deletion property. Let $e$ be a trivial orientable loop. Then by the definition of contraction of a loop, a spanning subgraph $F_3$ of $G/e$ corresponding to a subgraph $F_{2}\ni e$ of $G$ always has one more connected component than $F_{2}$, i.e. $k(F_2)=k(F_3)-1$. Hence,
\begin{align*}
    R_G(a,\mathbf{b},c) &= 
    \sum_{F\subseteq E(G)} W(F) \\
    &= \sum_{e\notin F_{1}, F_{1}\subseteq E(G)} W(F_1) + \sum_{e\in F_{2}, F_{2}\subseteq E(G)} W(F_2)\\
    &= \sum_{F_{1}\subseteq E(G-e)} W(F_1) + b_{e}\sum_{F_{3}\subseteq E(G/e)} a^{-1} W(F_3)\\
    &= R_{G-e}(a,\mathbf{b},c) + \frac{b_{e}}{a} R_{G/e}(a,\mathbf{b},c).
\end{align*}
\end{proof}

Analogous deletion-contraction properties hold for the twisted and loop Bollob\'{a}s-Riordan polynomials. We omit the details of these for brevity.

\begin{remark}
As in \cite[Remark.~2.6]{bradford2012arrow}, for a non-trivial orientable loop the evaluation of the {arrow Bollob\'{a}s-Riordan polynomial} at $a=1$ also satisfies a deletion-contraction relation. Namely, for a marked ribbon graph $G$ and $e\in E(G)$,
$$R_G(1,\mathbf{b},c)=R_{G-e}(1,\mathbf{b}_{\not=e},c) + 
b_e R_{G/e}(1,\mathbf{b}_{\not=e},c)\ ,
$$
where $\mathbf{b}_{\not=e}=\{b_{e'}\}_{e'\in E(G)\setminus e}$.
\end{remark}

\begin{remark}
Recall the classical Bollob\'{a}s-Riordan polynomial for ribbon graphs is multiplicative not only over disjoint union, but also over connected sum, i.e. $R_{G_1\# G_2} = R_{G_1}\cdot R_{G_2}$. Here the \textit{connected sum} $G_1\# G_2$ of two ribbon graphs $G_1,G_2$ is formed from $G_1\sqcup G_2$ by gluing a vertex from $G_1$ to a vertex from $G_2$ along arcs away from the edge discs. Clearly the resulting graph depends on the chosen vertices and arcs, but the result for $R_{G_1\# G_2}$ holds regardless of these choices. In our case, since we assume each marked ribbon graph has exactly one marking we can similarly define the \textit{pointed product} of two marked ribbon graphs as the connected sum along the marked segments of their marked vertices, being careful to respect the vertex-discs' orientations. This pointed product $G_1 \overline{\#} G_2$ is well-defined, and we have that $R_{G_1\overline{\#} G_2} = R_{G_1}\cdot R_{G_2}$ if we impose that $\Lambda_i\cdot \Lambda_j = \Lambda_{i+j}$ and $\Lambda^\prime_i\cdot \Lambda^\prime_j = \Lambda^\prime_{i+j}$, and for $i>j$ impose that $\Lambda^\prime_i\cdot \Lambda_j = \Lambda^\prime_{i-j}$ and $\Lambda_i\cdot \Lambda^\prime_j = \Lambda_{i-j}$. The correspondence of this fact to Remark \ref{rk:arrow_mult} follows from Theorem \ref{thm:thistlethwaite} in the next section.
\end{remark}

Finally we consider the behaviour of $R_G$ under partial duality.



\begin{proposition}\label{prop:BR_duals}
Let $F\subseteq E(G)$, and let $G^F$ be the partial dual of $G$ with respect to $F$. Then for $a=1$ the arrow Bollob\'{a}s-Riordan polynomial satisfies
\[
    R_G(1,\mathbf{b},c) = \big( \prod_{e\in F} b_e \big) \cdot R_{G^F}(1,\mathbf{b'},c)
\]
where $\mathbf{b'}$ is given by
\[
    b'_e = 
    \begin{cases}
    b_e &\text{ if } e\notin F\\
    1/b_e &\text{ if }e\in F.
    \end{cases}
\]
\end{proposition}

\begin{proof}
This proof is similar to that of the analogous statement for unmarked ribbon graphs, see \cite[Prop.~2.7]{bradford2012arrow}. There is a one-to-one correspondence between states of $G$ and states of $G'$, given by sending $H\subseteq E(G)$ to the symmetric difference $H' = H\Delta F \defeq (H\cup F)\setminus (H\cap F)$. We now claim the contribution of $H'$ to the right-hand side equals that of $H$ to the left-hand side. Indeed their boundaries correspond by construction of partial duality, so that the factors in $c,K_i,\Lambda_i^{(\prime)}$ in their contributions are equal. Finally their factors in $\mathbf{b}$ also agree, as
\[
    \big( \prod_{e\in F} b_e \big) \cdot \prod_{e'\in H'} b'_{e'}
    =
    \big( \prod_{e\in F} b_e \big) \cdot \prod_{e'\in H\setminus F} b_{e'} \cdot \prod_{e'\in F\setminus H} 1/b_{e'}
    =
    \prod_{e\in H} b_e
\]
\end{proof}


\section{Thistlethwaite theorems}\label{sec:Thistlethwaite}

\begin{definition}\label{def:thistlethwaite}
Let $K$ be a twisted or planar knotoid, and let $s$ be a state in the oriented state expansion of $K$. Then $G^s_K$ is defined to be the (punctured) marked ribbon graph constructed as follows:
\begin{itemize}
    \item The vertices of $G^s_K$ are given by the state components of $s$ by gluing a disk into each component, seeing the arc component as a marked vertex.
    \item The edges of $G^s_K$ correspond to the classical crossings of $K$. At each smoothing site in $s$ we place a small planar ribbon connecting the opposite arcs of that smoothing.
    \item Each edge is signed by `$+$' or `$-$' depending on whether the smoothing site at that edge incurs a factor $A$ or $B$ respectively in the oriented state expansion.
    \item Two arrow decorations are placed on the boundary of each edge in $G^s_K$ as follows: given an edge at a smoothing site in $s$, if this local smoothing respects the orientation of $K$ then we place an arrow on each arc of the edge that doesn't intersect a vertex. We place these arrows following the counter-clockwise orientation of the plane. Otherwise, if it is a disoriented smoothing, we place two arrows in the same way, but on the arcs of the edge that do intersect vertices of $G^s_K$.
    \item Bars in $K$ are left where they are to produce bar decorations in $G^s_K$, and similarly for the puncture in the planar case.
\end{itemize}
\end{definition}

See Figure \ref{fig:example_state} for an example of a twisted knotoid $K$, a state $s$, and the associated marked ribbon graph $G_K^s$.

\begin{figure}[ht]
    \centering
    \includegraphics[width=.65\linewidth]{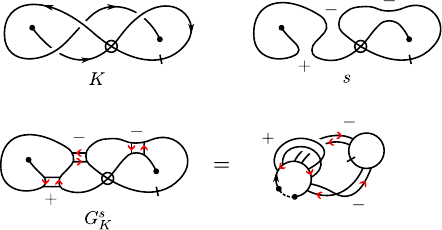}
    \caption{A marked ribbon graph $G_K^s$ associated to a state $s$ of a twisted knotoid $K$.}
    \label{fig:example_state}
\end{figure}

\begin{proposition}\label{prop:recovery}
Let $K$ be a twisted or planar knotoid and let $s$ be a state of $K$. Then a diagram for $K$ can be recovered from $G^s_K$.
\end{proposition}

\begin{proof}
Starting with a diagram for $G_K^s$, contract its edges along their length until they are rectangular and the arrows on them run along their boundaries in counter-clockwise direction, as in the bottom-left-hand side of Figure \ref{fig:example_state}. When these contractions create self-intersections of the diagram's vertices, mark these as virtual crossings. Now replace each contracted edge with a positive or negative crossing connecting the vertex arcs at that edge. The sign of each crossing is determined by the edge's sign and arrow information, by reversing the rules for placing these signs and arrows given in Definition \ref{def:thistlethwaite}. The resulting vertex arcs and crossings form a twisted knotoid diagram after splitting the marking on the marked vertex into two endpoints, which is a diagram for $K$.

To see that this procedure indeed correctly reconstructs $K$, independently of the chosen way of contracting edges, we note that it correctly reconstructs the crossing-region incidence information of $K$,  by construction. This information uniquely defines $K$, since it uniquely defines a knotoid diagram on a compact surface by capping off its boundary regions; recall Figure \ref{fig:thickening}.
\end{proof}

\begin{remark}
The proof of Proposition \ref{prop:recovery} essentially shows that the construction of Definition \ref{def:thistlethwaite} can be reversed. Consequently, for any signed arrow ribbon graph $G$ with appropriate arrow structure there exists a twisted knotoid diagram $K$ with a state $s$ such that $G=G_K^s$. Here the arrow structure is `appropriate' if all the arrows lie on the edges and each edge has exactly two arrows running along its boundary in the same direction 
and on opposite edge-arcs.
\end{remark}

\begin{proposition}\label{prop:state_duals}
Let $s,s'$ be states of a twisted or planar knotoid $K$. Then $G^s_K$ and $G^{s'}_K$ are partial duals of one another. Conversely any partial dual of $G^s_K$ is equal to $G^{s''}_K$ for some state $s''$ of $K$.
\end{proposition}

\begin{proof}
Let $C$ be the set of all crossings of $K$ at which $s$ differs from $s'$, and let $E(C)$ be the set of edges in $G_K^s$ corresponding to the crossings in $C$. Then $\left( G_K^s \right)^{E(C)} = G_K^{s'}$. Conversely, let $H=\left( G_K^s \right)^{E'}$ be a partial dual of $G^s_K$, where $E'\subseteq E(G^s_K)$. Then $H=G_K^{s''}$ where $s''$ is the state differing from $s$ precisely at the crossings corresponding to edges in $E'$.
\end{proof}

\begin{theorem}\label{thm:thistlethwaite}
Let $K$ be a virtual knotoid diagram and let $G_K^s$ be the marked ribbon graph decorated with arrows corresponding to a state $s$ of $K$. Then the arrow bracket of $K$ can be obtained from the arrow Bollob\'as-Riordan polynomial of $G_K^s$ as follows.
\[
    \langle K\rangle = \frac{A^{e_+}B^{e_-}}{d} R_{G_K^s}(1,\mathbf{b},d).
\]
Here $e_+$ and $e_-$ are the numbers of positive and negative edges in $s$, respectively, and the weights vector $\mathbf{b}$ is given by
\[
b_e = 
\begin{cases}
  B/A & \text{if $e$ is positive},\\
  A/B & \text{if $e$ is negative}.
\end{cases}
\]
Consequently, we have the following relation for the arrow polynomial.
\[
    \langle K\rangle_A = \frac{A^{e_+-e_-}}{-A^2-A^{-2}} R_{G_K^s}(1,\mathbf{b}\vert_{B=A^{-1}},-A^2-A^{-2}).
\]
\end{theorem}

\begin{proof}
Note that the right-hand side is independent of $s$ due to Propositions \ref{prop:BR_duals} and \ref{prop:state_duals}. The second identity follows immediately from the first by substituting $B=A^{-1}$ and $d=(-A^2-A^{-2})$. To show the first identity, consider the one-to-one correspondence between states of $K$ and of $G_K^s$ given by associating to a state $s'$ of $K$ the subgraph with edge set $F(s')\subseteq E(G_K^s)$ containing those edges where $s'$ differs from $s$. We show that the contribution of $s'$ to $\langle K\rangle$ is equal to contribution of $F(e')$ to the right-hand side. 

The boundary components of $F(s')$ correspond exactly to the state components of $s'$ by construction of $G_K^s$. Indeed, recall that the edges in $G_K^s$ attach at the smoothing sites of $s$. If an edge is not in $F(s')$, then the smoothing site in $s$ associated to that edge is smoothed the same in $s'$, and removing the edge locally yields the correct boundary in $F(s')$. If an edge is in $F(s')$ then the associated smoothing site is smoothed oppositely in $s'$, and including the edge locally corrects the boundary to follow the smoothing in $s'$. From this it follows that both contributions have equal factors of $d$, $K_i$, and $\Lambda^{(\prime)}_i$.

So to see that the contributions are equal it suffices to show they have the same powers of $A$ and $B$. The powers of $A$ and $B$ in the contribution of $s'$ are just the numbers of $A$-smoothings and $B$-smoothings in $s'$; call these numbers $p_a(s')$ and $p_b(s')$ respectively. The power of $A$ in the contribution from $F(s')$ is
\[
    e_+(G_K^s) - e_+(F(s')) + e_-(F(s')) = e_+(G_K^s \setminus F(s')) + e_-(F(s')) = p_a(s').
\]
Here the final equality follows from the definition of $F(s')$. Similarly the power of $B$ in the contribution from $F(s')$ is
\[
    e_-(G_K^s) - e_-(F(s')) + e_+(F(s')) = e_-(G_K^s \setminus F( s')) + e_+(F(s')) = p_b(s').
\]
\end{proof}

\begin{remark}\label{rk:thistlethwaite}
Theorem \ref{thm:thistlethwaite} is a generalization of Thistlethwaite's theorem \cite{thistlethwaite1987}, which can be phrased as asserting that the Kauffman bracket polynomial of a link diagram can be recovered from the Bollob\'{a}s-Riordan polynomial of an associated ribbon graph. Recall that the oriented state expansion from Figure \ref{fig:expansion} is a generalization of the Kauffman skein relation, and that the Kauffman bracket may be obtained from the arrow bracket by omitting all arrows and orientations. Hence as a corollary to Theorem \ref{thm:thistlethwaite} we also obtain a generalization of Thistlethwaite's theorem for the Kauffman bracket of knotoids \cite{turaev2012}, by dropping all arrows and orientations.
\end{remark}

Similarly Thistlethwaite's theorem extends to the twisted arrow- and loop arrow polynomials:

\begin{theorem}\label{thm:thistlethwaite_t}
Let $K$ be a twisted knotoid diagram and let $G_K^s$ be the marked ribbon graph decorated with bars and arrows corresponding to a state $s$ of $K$. Then
\[
    \langle K\rangle^t = \frac{A^{e_+}B^{e_-}}{d} R^t_{G_K^s}(1,\mathbf{b},d).
\]
where $e_\pm$ and $\mathbf{b}$ are as before. Consequently
\[
    \langle K\rangle^t_A = \frac{A^{e_+-e_-}}{-A^2-A^{-2}} R^t_{G_K^s}(1,\mathbf{b}\vert_{B=A^{-1}},-A^2-A^{-2}).
\]
\end{theorem}

\begin{proof}
Analogous to the proof of Theorem \ref{thm:thistlethwaite}.
\end{proof}

\begin{theorem}\label{thm:thistlethwaite_l}
Let $K$ be a planar knotoid diagram, and $G_K^s$ the punctured marked ribbon graph decorated with arrows corresponding to a state $s$ of $K$. Then
\[
    \langle K\rangle^\ell = \frac{A^{e_+}B^{e_-}}{d} R^\ell_{G_K^s}(1,\mathbf{b},d),
\]
where $e_\pm$ and $\mathbf{b}$ are as before. Consequently
\[
    \langle K\rangle^\ell_A = \frac{A^{e_+-e_-}}{-A^2-A^{-2}} R^\ell_{G_K^s}(1,\mathbf{b}\vert_{B=A^{-1}},-A^2-A^{-2}).
\]
\end{theorem}

\begin{proof}
This proof is again analogous to the proof of Theorem \ref{thm:thistlethwaite}, where we take care to note that the puncture data in each state $s'$ of $K$ is correctly retrieved from $F(s')$, again by construction of $G_K^s$.
\end{proof}

\begin{corollary}
Similarly to Remark \ref{rk:thistlethwaite}, Theorems \ref{thm:thistlethwaite_t} and \ref{thm:thistlethwaite_l} imply generalizations of Thistlethwaite's theorem for the Kauffman bracket of twisted knotoids and the Turaev loop bracket \cite{turaev2012} of planar knotoids respectively.
\end{corollary}

\begin{example}
As an example we will verify the Thistlethwaite theorems for some small diagrams. In this example we will focus on the arrow polynomials after evaluating $B=A^{-1}$, $d=(-A^2-A^{-2})$ but will use the shorthand notation $d = -A^2-A^{-2}$. First, let $K_1$ be the nontrivial spherical knotoid with two negative crossings; see Figure \ref{fig:theorem_ex1}. Then it is easy to check that
\[
    \langle K_1 \rangle_A = A^{-2} + (1-A^4)\Lambda_1.
\]

\begin{figure}[ht]
    \centering
    \includegraphics[width=.75\linewidth]{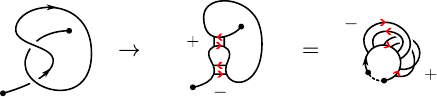}
    \caption{Associating a ribbon graph to a state of $K_1$.}
    \label{fig:theorem_ex1}
\end{figure}

Figure \ref{fig:theorem_ex1} also depicts a marked ribbon graph $G_{K_1}^s$ associated to a state $s$ of $K_1$. In fact this ribbon graph is exactly the graph $G$ from Example \ref{ex:graph_poly}, with a negative label for $e_1$ and a positive label for $e_2$ so that $b_{e_1}=A^2$, $b_{e_2}=A^{-2}$. So from Example \ref{ex:graph_poly} we find
\[
    R_{G^s_{K_1}}(1,\mathbf{b},d) = dA^{-2} + (2d+A^2d^2)\Lambda_1.
\]
Hence using that $e_+(s)=e_-(s)=1$ we compute
\[
    \frac{A^{e_+-e_-}}{d} R_{G_K^s}(1,\mathbf{b},d) = \frac{R_{G_K^s}(1,\mathbf{b},d)}{d} = A^{-2} + (2+A^2d)\Lambda_1 = \langle K_1 \rangle_A,
\]
as expected.

Next let $K_2$ be the twisted knotoid from Example \ref{ex:twisted}, for which we have seen that
\[
    \langle K_2\rangle_A^t = \left(A^2 + 1 + A^{-2}\right) + d\Lambda_1 K_1.
\]
The marked ribbon graph $G_{K_2}^s$ corresponding to the state $s$ with only positive smoothings is given in Figure \ref{fig:theorem_ex2}.

\begin{figure}[ht]
    \centering
    \includegraphics[width=.75\linewidth]{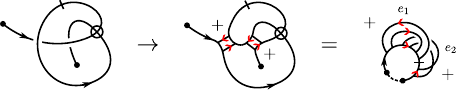}
    \caption{Associating a ribbon graph to a state of $K_2$.}
    \label{fig:theorem_ex2}
\end{figure}

To compute $R_{G_{K_2}^s}^t(1,\mathbf{b},d)$ we find the contributions of the four subsets $F\subseteq \{e_1,e_2\}$, where $e_1,e_2$ are the edges of $G_{K_1}^s$ as indicated in Figure \ref{fig:theorem_ex2}. The contribution from $F=\emptyset$ is $d$, since this subgraph has no edges, only one boundary component, and all markings on it cancel. The contribution from $\{e_1\}$ is $A^{-2}d^2K_1\Lambda_1$, as there are two boundary components, each with two arrows that do not cancel, and one edge contributing a factor $b_{e_1}=A^{-2}$. Finally the subgraphs for $\{e_2\}$ and $\{e_1,e_2\}$ both have one boundary component without arrows in their reduced form, so that their contributions are $A^{-2}d$ and $A^{-4}d$ respectively. So in total,
\[
    R_{G^s_{K_2}}^t(1,\mathbf{b},d) = d + A^{-2}d^2 K_1 \Lambda_1 + A^{-2}d + A^{-4}d.
\]
Using that $e_+(s)=2$ and $e_-(s)=0$ we therefore find
\begin{align*}
    \frac{A^{e_+-e_-}}{d} R^t_{G_{K_2}^s}(1,\mathbf{b},d) &= \frac{A^2}{d} \left( d + A^{-2}d^2 K_1 \Lambda_1 + A^{-2}d + A^{-4}d \right)\\
    &= A^2 + dK_1\Lambda_1 + 1 + A^{-2}\\
    &= \langle K_2 \rangle_A^t.
\end{align*}

Finally to verify a small example of the planar Thistlethwaite theorem, let $K_3$ be the non-trivial 1-crossing planar knotoid from Figure \ref{fig:1crossing}. It is immediate that
\[
    \langle K_3 \rangle_A^\ell = A + A^{-1}(\Lambda_0)^1.
\]
The ribbon graph $G_{K_3}^s$ corresponding to the state with a negative smoothing is given in Figure \ref{fig:theorem_ex3}.

\begin{figure}[ht]
    \centering
    \includegraphics[width=.65\linewidth]{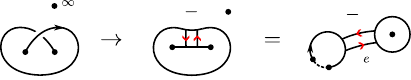}
    \caption{Associating a ribbon graph to a state of $K_3$.}
    \label{fig:theorem_ex3}
\end{figure}

The contribution to $R_{G^s_{K_3}}^\ell$ of $F=\emptyset$ is $d(\Lambda_0)^1$, as there are two boundary components all of whose arrows cancel, and $\ell=1$. The contribution from $F=\{e\}$ is $A^2d$ as there is one boundary component and one negative edge. So in total
\[
    R_{G^s_{K_3}}^\ell(1,\mathbf{b},d) = A^2d + d(\Lambda_0)^1,
\]
and using that $e_+(s)=0$ and $e_-(s)=1$ we find
\[
    \frac{A^{e_+-e_-}}{-A^2-A^{-2}} R^\ell_{G_{K_3}^s}(1,\mathbf{b},d) = \frac{A^{-1}}{d}(A^2d+d(\Lambda_0)^1) = A + A^{-1}(\Lambda_0)^1 = \langle K_3\rangle^\ell_A,
\]
as expected.
\end{example}



\section{Linkoids}\label{sec:linkoids}

As a final generalization, we briefly discuss the extension of the theory covered so far to the case of `linkoids' \cite{gabrovvsek2023}, which are knotoids with multiple components. For the sake of brevity we restrict out attention in this section to classical diagram on the sphere. As we shall see, in the multi-component case this still gives rise to an interesting extension of the bracket polynomial.

\begin{definition}
\cite{gabrovvsek2023} A \emph{linkoid diagram} on $S^2$ is an immersion of finitely many copies of $I$ and $S^1$ into $S^2$, all whose singularities are crossings with over/under-crossing information. A \emph{linkoid} is an equivalence class of such diagrams, considered up to the same equivalence relation as for knotoids. A component of a linkoid is \emph{knotoidal} if it has endpoints. A linkoid with a single knotoidal component is also called a \emph{multiknotoid} \cite{turaev2012}.
\end{definition}

\begin{figure}[ht]
    \centering
    \includegraphics[width=.8\linewidth]{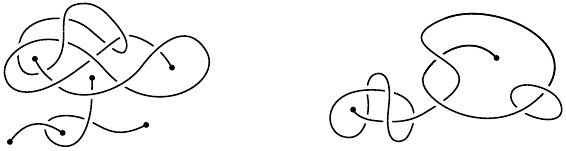}
    \caption{A linkoid (left) and a multiknotoid (right).}
    \label{fig:linkoids}
\end{figure}

We first consider what type of states may arise in the oriented state expansion of a linkoid diagram. In fact the case is quite different from that for knotoids, since several results for knotoid states relied on there being exactly one knotoidal component. In particular we have that Lemmas \ref{lm:even_arrows} and \ref{lm:classical_arrows} do not generalize to linkoids, as is shown by the following counter-example:

\begin{example}
Let $L$ and $S$ be the linkoid and its state depicted in Figure \ref{fig:counterex}. Then we see that the knotoidal components of $S$ each have an odd number of arrows, and that the circular component of $S$ has two arrows on its fully reduced form.
\end{example}

\begin{figure}[ht]
    \centering
    \includegraphics[width=.7\linewidth]{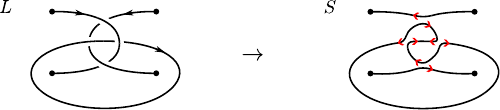}
    \caption{A counterexample to Lemmas \ref{lm:even_arrows} and \ref{lm:classical_arrows} for linkoids.}
    \label{fig:counterex}
\end{figure}

In light of this counterexample, to define an arrow polynomial of linkoids we should reconsider the kinds of reduced states that can arise from the oriented state expansion. Other than noting that the linkoid case allows for marked ribbon graphs with several marked vertices, this comes down to classifying the reduced arc components with an odd number of arrows. Since the orientation of a state component flips at every arrow, we find that the endpoints of such a component in $G^s_L$ must have both of the same type; either both heads or both tails. It is further easy to see that such an arc state is fully reduced if and only if the arrows on it are alternating.

As such, the fully reduced arc components with an odd number of arrows come in two types: those with two head-endpoints and those with two tail-endpoints. We shall denote these reduced components by $\Lambda^h_{i/2}$ and $\Lambda^t_{i/2}$ respectively, where $i$ is the number of arrows on the component. In total, there are then four types of fully reduced arc component: those with integer indices, $\Lambda_i$ and $\Lambda_i'$, and those with half-integer indices, $\Lambda^h_{i/2}$ and $\Lambda^t_{i/2}$.

\begin{remark}
Since we are in the classical setting we do not need to be concerned with the presence of bars, so that $\Lambda_i$ and $\Lambda_i'$ are indeed distinct state components. It is interesting to note, however, that passing a bar over the entirety of a copy of $\Lambda^h_{i/2}$ or $\Lambda^t_{i/2}$ has no effect on these state components, meaning they remain distinct in the twisted setting.
\end{remark}

Next we consider the loop components in a linkoid state. These are much the same as before, except that since we are in the classical case we can attach a sort of puncture information to each loop component. Namely, in the classical case an arc component in a state behaves much like a puncture, as we do not have the move $V0$ (recall Figure \ref{fig:moves}) at our disposal. Seeing our loop component as a Jordan curve with an interior and exterior, it is clear such an arc cannot be moved from interior to exterior or vice versa using Reidemeister moves. Thus we can count the arc components in the interior and exterior of a loop component, and attach this information to the component. Call the results of these counts $n_i$ for the interior and $n_e$ for the exterior. As we have assumed the spherical case, there is no canonical choice for which region is the interior and which the exterior of the loop. To remedy this we simply index a loop component with $\min(n_i,n_e)$, which lies between $0$ and $\lfloor \kappa/2 \rfloor$ where $\kappa$ is the number of knotoidal components.

\begin{definition}\label{def:linkoid_poly}
Let $L$ be a spherical linkoid diagram. The \emph{arrow bracket} of $L$, denoted $\langle L\rangle^m$, is defined to be the polynomial given by carrying out the oriented state expansion on $L$ and summing up the values of the resulting states analogously to Definition \ref{def:arrow}, resulting in a polynomial $\langle K\rangle^m \in \mathbb{Z}[A,B,d, \{K_i^\ell\}_{i\in\mathbb{N}, 0\leq \ell \leq \lfloor \kappa/2 \rfloor}, \{\Lambda_i, \Lambda'_i, \Lambda^h_{i/2}, \Lambda^t_{i/2}\}_{i\in \mathbb{N}} ]$. Here $\kappa$ is the number of knotoidal components in $L$. The \emph{arrow polynomial} $\langle L\rangle^m_A$ of $L$ is obtained from $\langle L\rangle ^m$ by the substitution $B=A^{-1}, d=(-A^2-A^{-2})$.
\end{definition}

\begin{remark}
In keeping track of the number of arc components contained inside each loop component we could go further: one could instead keep track of the number of fully reduced arc components of each type and index separately, since for each state the total number of arcs of each type and index is known. To avoid overly cumbersome notation we choose not to do so here.
\end{remark}


\begin{definition}
The \emph{arrow Bollob\'{a}s-Riordan polynomial} for decorated marked ribbon graphs with several marked vertices is defined by
\[
    R^m_G(a,\mathbf{b},c) = \sum_{F\subseteq E(G)} a^{k(F)} \left( \prod_{e\in F} b_e \right) c^{\text{bc}(F)} \left( \prod_{f\in\partial^c(F)} K_{i(f)/2}^{\ell(f)} \right) \left( \prod_{f\in\partial^a(F)} \Lambda^{\sigma(f)}_{i(f)/2} \right).
\]
Here $\partial^c(F)$ and $\partial^a(F)$ denote the set of circular boundary components and arc boundary components of $F$, respectively, and $i(f)$ denotes the number of arrows on a boundary component $f$. Moreover $\ell(f)$ denotes the minimum of the numbers of arc components in the interior and exterior of $f\in \partial^c(F)$. Finally $\sigma(f)$ is either an empty symbol or one of the symbols $\{\prime,h,t\}$, based on the endpoints of $f\in \partial^a(F)$.
\end{definition}

\begin{remark}
The arrow Bollob\'{a}s-Riordan polynomial for linkoids satisfies a deletion-contraction relation analogous to Proposition \ref{prop:del_cont}. However in the presence of multiple knotoidal components we have no canonical pointed product as we do for knotoids.
\end{remark}

Finally, we state the Thistlethwaite theorem for spherical linkoids, the proof of which is analogous to that of Theorem \ref{thm:thistlethwaite}:

\begin{theorem}
Let $L$ be a linkoid diagram and let $G_L^s$ be the decorated marked ribbon graph corresponding to a state $s$ of $K$. Then
\[
    \langle K\rangle^m_A = \frac{A^{e_+}B^{e_-}}{d} R^m_{G_K^s}(1,\mathbf{b},d).
\]
where $e_\pm$ and $\mathbf{b}$ are as in Theorem \ref{thm:thistlethwaite}.
\end{theorem}


\noindent \textbf{Acknowledgements:} We would like to thank Iain Moffatt, Kerri Morgan, and Graham Farr for co-organizing the Workshop on Uniqueness and Discernment in Graph Polynomials at the MATRIX institute on 16-27 October 2023, where this work was initiated. We are very grateful for the support and productive research environment provided by MATRIX. Thanks also go to the participants for valuable discussions.



\begin{thebibliography}{10}

\bibitem{barbensi2018double}
Agnese Barbensi, Dorothy Buck, Heather~A Harrington, and Marc Lackenby.
\newblock Double branched covers of knotoids.
\newblock {\em arXiv preprint arXiv:1811.09121}, 2018.

\bibitem{bourgoin2008}
Mario~O Bourgoin.
\newblock Twisted link theory.
\newblock {\em Algebraic \& Geometric Topology}, 8(3):1249--1279, 2008.

\bibitem{bradford2012arrow}
Robert Bradford, Clark Butler, and Sergei Chmutov.
\newblock Arrow ribbon graphs.
\newblock {\em Journal of Knot Theory and its Ramifications}, 21(13):1240002, 2012.

\bibitem{chmutov-2009}
Sergei Chmutov.
\newblock Generalized duality for graphs on surfaces and the signed {B}ollob\'as-{R}iordan polynomial.
\newblock {\em Journal of Combinatorial Theory, Ser. B}, 99(3):617--638, 2009.

\bibitem{chmutov-pak2007}
Sergei Chmutov and Igor Pak.
\newblock The {K}auffman bracket of virtual links and the {B}ollob\'as-{R}iordan polynomial.
\newblock {\em Moscow Mathematical Journal}, 7(3):409--418, 2007.

\bibitem{chmutov-voltz2008}
Sergei Chmutov and Jeremy Voltz.
\newblock Thistlethwaite's theorem for virtual links.
\newblock {\em Journal of Knot Theory and Its Ramifications}, 17(10):1189֭--1198, 2008.

\bibitem{DFKLS}
Oliver Dasbach, David Futer, Efstratia Kalfagianni, Xiao-Song Lin, and Neal Stoltzfus.
\newblock The {J}ones polynomial and graphs on surfaces.
\newblock {\em Journal of Combinatorial Theory, Ser.B}, 98:384--399, 2008.

\bibitem{deng2022}
Qingying Deng.
\newblock One conjecture on cut points of virtual links and the arrow polynomial of twisted links.
\newblock {\em Journal of Knot Theory and Its Ramifications}, 31(10):2250066, 2022.

\bibitem{dorier2018knoto}
Julien Dorier, Dimos Goundaroulis, Fabrizio Benedetti, and Andrzej Stasiak.
\newblock Knoto-{ID}: a tool to study the entanglement of open protein chains using the concept of knotoids.
\newblock {\em Bioinformatics}, 34(19):3402--3404, 2018.

\bibitem{ellis2013}
Joanna~A Ellis-Monaghan and Iain Moffatt.
\newblock {\em Graphs on surfaces: dualities, polynomials, and knots}, volume~84.
\newblock Springer, 2013.

\bibitem{gabrovvsek2023}
Bo{\v{s}}tjan Gabrov{\v{s}}ek and Neslihan G{\"u}g{\"u}mc{\"u}.
\newblock Invariants of multi-linkoids.
\newblock {\em Mediterranean Journal of Mathematics}, 20(3):165, 2023.

\bibitem{goundaroulis2019systematic}
Dimos Goundaroulis, Julien Dorier, and Andrzej Stasiak.
\newblock A systematic classification of knotoids on the plane and on the sphere.
\newblock {\em arXiv preprint arXiv:1902.07277}, 2019.

\bibitem{goundaroulis2020knotoids}
Dimos Goundaroulis, Julien Dorier, and Andrzej Stasiak.
\newblock Knotoids and protein structure.
\newblock {\em Topol. Geom. Biopolym}, 746:185, 2020.

\bibitem{goundaroulis2017}
Dimos Goundaroulis, Neslihan G{\"u}g{\"u}mc{\"u}, Sofia Lambropoulou, Julien Dorier, Andrzej Stasiak, and Louis Kauffman.
\newblock Topological models for open-knotted protein chains using the concepts of knotoids and bonded knotoids.
\newblock {\em Polymers}, 9(9):444, 2017.

\bibitem{gugumcu2017}
Neslihan G{\"u}g{\"u}mc{\"u} and Louis~H Kauffman.
\newblock New invariants of knotoids.
\newblock {\em European Journal of Combinatorics}, 65:186--229, 2017.

\bibitem{hudson1964combinatorial}
John~FP Hudson and Eric~Christopher Zeeman.
\newblock On combinatorial isotopy.
\newblock {\em Publications Math{\'e}matiques de l'IH{\'E}S}, 19:69--94, 1964.

\bibitem{kauffman1987}
Louis~H Kauffman.
\newblock State models and the {J}ones polynomial.
\newblock {\em Topology}, 26(3):395--407, 1987.

\bibitem{kauffman1989}
Louis~H Kauffman.
\newblock A {T}utte polynomial for signed graphs.
\newblock {\em Discrete Appl.~Math.}, 25:105--127, 1989.

\bibitem{kauffman2012}
Louis~H Kauffman.
\newblock Introduction to virtual knot theory.
\newblock {\em Journal of Knot Theory and Its Ramifications}, 21(13):1240007, 2012.

\bibitem{kauffman2021}
Louis~H Kauffman.
\newblock Virtual knot theory.
\newblock {\em Encyclopedia of Knot Theory}, page 261, 2021.

\bibitem{moltmaker2022}
Wout Moltmaker.
\newblock Framed knotoids and their quantum invariants.
\newblock {\em Communications in Mathematical Physics}, pages 1--27, 2022.

\bibitem{moltmaker2023}
Wout Moltmaker and Roland van~der Veen.
\newblock New quantum invariants of planar knotoids.
\newblock {\em Communications in Mathematical Physics}, pages 1--28, 2023.

\bibitem{panagiotou2020knot}
Eleni Panagiotou and Louis~H Kauffman.
\newblock Knot polynomials of open and closed curves.
\newblock {\em Proceedings of the Royal Society A}, 476(2240):20200124, 2020.

\bibitem{paugh-wu-zhang2021}
Joseph Paugh, Justin Wu, and Gavin Zhang.
\newblock A twisted {T}histlethwaite theorem.
\newblock A talk at the {\it Young Mathematician Conference}. Available at {\tt https://people.math.osu.edu/chmutov.1/wor-gr-su21/Paugh-slides.pdf}, 2021.

\bibitem{thistlethwaite1987}
Morwen Thistlethwaite.
\newblock A spanning tree expansion for the {J}ones polynomial.
\newblock {\em Topology}, 26:297--309, 1987.

\bibitem{turaev2012}
Vladimir Turaev.
\newblock Knotoids.
\newblock {\em Osaka Journal of Mathematics}, 49(1):195--223, 2012.

\end{thebibliography}

\end{document}